\documentclass[a4paper,12pt]{article}
\usepackage{amsfonts,amsmath,amssymb,amsthm,verbatim,placeins,color,pgfplots,comment,placeins,caption,subcaption,enumerate,bm}
\usepackage{graphicx}
\usepackage{geometry}
\usepackage[pdfauthor={DG},pdfstartview={FitH},colorlinks=true,citecolor=blue]{hyperref}
\usepackage[english]{babel}
\usepackage[sort&compress]{natbib}

\newtheorem{theorem}{Theorem}[section]

\newtheorem{lemma}{Lemma}[section]

\theoremstyle{remark}

\theoremstyle{definition}

\newtheorem*{conjecture}{Conjecture}

\theoremstyle{remark}
\newtheoremstyle{myremark}{}{}{\color{blue}\small}{}{\color{blue}\bfseries}{}{ }{}

\theoremstyle{myremark}

\newcommand{\E}{\mathbb{E}} 
\newcommand{\Var}{\operatorname{Var}} 
\renewcommand{\Re}{\operatorname{Re}} 
\newcommand{\R}{{\mathbb R}}
\newcommand{\N}{{\mathbb N}}

\DeclareMathOperator{\sign}{sign}
\newcommand{\mff}{\mathfrak{f}}
\newcommand\1{\mathbf{1}}
\newcommand\X{\mathcal{X}}

\makeatletter
\renewcommand*{\@fnsymbol}[1]{\ensuremath{\ifcase#1\or *\or \mathparagraph\or \ddagger\or
        \mathsection\or \mathparagraph\or \|\or **\or \dagger\dagger
        \or \ddagger\ddagger \else\@ctrerr\fi}}
\makeatother

\usepackage{xcolor}
\usepackage{graphicx}

\begin{document}
\setcounter{page}{1000}
\thispagestyle{empty}
\vspace*{10cm}
\noindent This is a preprint of the paper:\\

\noindent Grahovac, D., Leonenko, N.~N., Taqqu, M.~S.~(2019) Limit theorems, scaling of moments and intermittency for integrated finite variance supOU processes, \emph{Stochastic Processes and their Applications}, \textbf{}, accepted for publication.\footnote{\copyright 2019. This manuscript version is made available under the CC-BY-NC-ND 4.0 license \url{http://creativecommons.org/licenses/by-nc-nd/4.0/}}\\

\noindent URL: \url{https://www.sciencedirect.com/science/article/abs/pii/S0304414918303363?via%3Dihub}\\
DOI: \url{https://doi.org/10.1016/j.spa.2019.01.010}
\newpage
\setcounter{page}{1}

\renewcommand*{\thefootnote}{\fnsymbol{footnote}}

\begin{center}
\Large{\textbf{Limit theorems, scaling of moments and intermittency for integrated finite variance supOU processes}}\\
\bigskip
\bigskip
Danijel Grahovac$^1$\footnote{dgrahova@mathos.hr}, Nikolai N.~Leonenko$^2$\footnote{LeonenkoN@cardiff.ac.uk}, Murad S.~Taqqu$^3$\footnote{murad@bu.edu}\\
\end{center}

\bigskip
\begin{flushleft}
\footnotesize{
$^1$ Department of Mathematics, University of Osijek, Trg Ljudevita Gaja 6, 31000 Osijek, Croatia\\
$^2$ School of Mathematics, Cardiff University, Senghennydd Road, Cardiff, Wales, UK, CF24 4AG}\\
$^3$ Department of Mathematics and Statistics, Boston University, Boston, MA 02215, USA
\end{flushleft}

\bigskip

\textbf{Abstract: }
Superpositions of Ornstein-Uhlenbeck type (supOU) processes provide a rich class of stationary stochastic processes for which the marginal distribution and the dependence structure may be modeled independently. We show that they can also display intermittency, a phenomenon affecting the rate of growth of moments. To do so, we investigate the limiting behavior of integrated supOU processes with finite variance. After suitable normalization four different limiting processes may arise depending on the decay of the correlation function and on the characteristic triplet of the marginal distribution. To show that supOU processes may exhibit intermittency, we establish the rate of growth of moments for each of the four limiting scenarios. The rate change indicates that there is intermittency, which is expressed here as a change-point in the asymptotic behavior of the absolute moments.

\bigskip

\section{Introduction}

The Ornstein-Uhlenbeck (OU) processes driven by L\'{e}vy noise and their superpositions (supOU processes) where constructed with the aim of modeling key stylized features of observational series from finance and turbulence. The goal is to find models with analytically and stochastically tractable correlation structure displaying either weak or strong dependence and also having marginal distributions that are infinitely divisible and hence related to both Gaussian and Poisson processes. The supOU processes are particularly relevant in finance and the statistical theory of turbulence. They have applications in environmental studies, ecology, meteorology, geophysics, biology, see e.g.~\cite{barndorff2015recent} and the references therein. A key characteristics are the stochastic representations using L\'{e}vy basis (i.e.~an independently scattered and infinitely divisible random measure). The attractive feature of supOU processes is that they allow the marginal distribution and the dependence structure to be modeled independently from each other. Moreover, they offer a flexible choice of different forms of correlation functions. In particular, the class of finite variance stationary supOU processes contains examples where the correlation function $r( \tau) $ decreases like a power function as the lag increases, more precisely where
\begin{equation}\label{regvarofACR}
r( \tau) \sim \Gamma(1+\alpha) \ell(\tau) \tau^{-\alpha}, \qquad \text{ as } \tau \to \infty,
\end{equation}
for some slowly varying function $\ell$ and $\alpha>0$ (see Section \ref{secPre} for more details). Hence, if $\alpha \in (0,1)$, the correlation function is not integrable, and the supOU process exhibits long-range dependence (long memory or strong dependence). The volume of \cite{doukhan2003} contains surveys of the field.  Note that \cite{metzler1999anomalous} reported a Mittag-Leffler decay in the autocorrelation function of the velocity of a particle in anomalous diffusion. Such a decay can be modeled by the correlation function of supOU processes (see \cite[Example 4]{barndorff2005spectral}).

An exciting area of applications of supOU processes is financial econometrics, in particular the stochastic volatility models, see \cite{barndorff1997processes}, \cite{barndorff2001non} and the references therein. In this setting the \textit{integrated supOU process} \eqref{integratedsupOU} defined below represents the integrated volatility (see e.g.~\cite{barndorff2013multivariate}), hence its limiting behavior is particularly important. The limit theorems developed in the paper can be used for statistical inference based on the generalized method of moments or method of minimum contrast (see e.g.~\cite{stelzer2015moment}). Our results also indicate that to obtain the limiting behavior one has to know or estimate the behavior of the L\'evy measure near the origin which can be challenging (see \cite{Belomestny2015} and the references therein). Just recently in astrophysics the authors of \cite{kelly2013active} (see also the references therein) use the supOU processes to asses the mass of black hole. They used heuristic arguments to estimate parameters of the model under long-range dependence. But to develop mathematical procedures, one needs precise limit theorems as those obtained in this paper.

In this paper we focus on supOU processes $X=\{X(t), \, t \geq 0\}$ having finite variance and investigate the limiting behavior of the  integrated process $X^*=\{X^*(t), \, t \geq 0\}$ where
\begin{equation}\label{integratedsupOU}
X^*(t) = \int_0^t X(s) ds.
\end{equation}
A long quest has preceded the results presented here. The pioneering work of Barndorff-Nielsen \cite{bn2001} already contained a limit theorem corresponding to a specific triangular scheme. Non-central limit theorems with convergence to fractional Brownian motion appeared in \cite{barndorff2005burgers,leonenko2005convergence}. From the results presented here, it is now clear that these do not hold in general and that they depend on the rate of growth of the moments of the integrated process $X^*$, see also \cite{GLST2017Arxiv}, \cite{GLST2016JSP}. 

We focus here on how an unusual rate of growth of the integrated process $X^*(t)$ can affect limit theorems. We refer to this rate as intermittency. There is no unique definition of intermittency. It is a relative concept and its meaning depends on the particular setting under investigation. It refers in general to an unusual moment behavior. It is of major importance in many fields of science, such as the study of rain and cloud studies and other aspects of environmental science, on relation to nanoscale emitters, magnetohydrodynamics, liquid mixtures of chemicals and physics of fusion plasmas, see e.g.~\cite{zel1987intermittency}. Another area of possible application is turbulence. In turbulence, the velocities or velocity derivatives (or differences) under a large Reynolds number could be modeled, as is done here, with infinitely divisible distributions, they allow long range dependence and there seems to exist a kind of switching regime between periods of relatively small random fluctuation and period of ``higher'' activity. This phenomenon is also referred to as intermittency, see e.g.~\cite[Chapter 8]{frisch1995turbulence} or \cite{zel1987intermittency}.

We use the following definition of intermittency. Let $Y=\{Y(t),\, t \geq 0\}$ be a stochastic process. We shall measure the rate of growth of moments by the \textit{scaling function}, defined by
\begin{equation}\label{deftauY}
\tau_Y(q) = \lim_{t\to \infty} \frac{\log \E |Y(t)|^q}{\log t},
\end{equation}
assuming the limit in \eqref{deftauY} exists and is finite. The range of moments $q$ can be infinite or finite, that is $q \in (0,\overline{q}(Y))$, where 
\begin{equation*}
\overline{q}(Y) = \sup \{ q >0 :\E|Y(t)|^q < \infty  \ \forall t\}.
\end{equation*}
It has been shown in \cite{GLST2017Arxiv} that for a non-Gaussian integrated supOU process $X^*$ with marginal distribution having exponentially decaying tails and correlation function satisfying \eqref{regvarofACR}, the scaling function is
\begin{equation}\label{tauX*}
\tau_{X^*}(q)=q-\alpha
\end{equation}
for a certain range of $q$. This implies that the function
\begin{equation*}
q \mapsto \frac{\tau_{X^*}(q)}{q} = 1-\frac{\alpha}{q}
\end{equation*}
is strictly increasing, a property referred to as \textit{intermittency}. Recently, the term \textit{additive intermittency} has also been used (see \cite{chong2018almost}).

To see why this behavior of the scaling function is unexpected and interesting, recall that by Lamperti's theorem (see, for example, \cite[Theorem 2.8.5]{pipiras2017long}), the limits of normalized processes are necessarily self-similar, that is, if
\begin{equation}
\left\{ \frac{X^*(Tt)}{A_T} \right\} \overset{fdd}{\to} \left\{ Z(t) \right\}, \label{limitform}
\end{equation}
holds with convergence in the sense of convergence of all finite dimensional distributions as $T \to \infty$, then $Z$ is $H$-self-similar for some $H>0$, that is, for any constant $c>0$, the finite dimensional distributions of $Z(ct)$ are the same as those of $c^H Z(t)$. Brownian motion for example is self-similar with $H=1/2$. Moreover, the normalizing sequence is of the form $A_T=\ell(T) T^H$ for some $\ell$ slowly varying at infinity. For self-similar process, the moments evolve as a power function of time since $\E|Z(t)|^q=\E|Z(1)|^q t^{Hq}$ and therefore the scaling function of $Z$ is $\tau_Z(q)=Hq$. Hence \eqref{tauX*} does not hold for self-similar processes. But it may not hold either for the process $X^*$ in \eqref{limitform} because one would expect that
\begin{equation}
\frac{\E| X^*(Tt)|^q}{A_T^q} \to \E |Z(t)|^q, \quad \forall t \geq 0, \label{limitformmom}
\end{equation}
and therefore that $\E|X^*(t)|^q$ grows roughly as $t^{Hq}$ when $t\to \infty$. Since \eqref{tauX*} implies that $\E|X^*(t)|^q$ grows roughly as $t^{q-\alpha}$, we conclude that the intermittency of the process $X^*$ contradicts \eqref{limitform} or \eqref{limitformmom} or both. See \cite{GLST2017Arxiv} for the precise statements.

This paper has two main goals:
\begin{enumerate}[(i)]
	\item to establish limit theorems in the form \eqref{limitform} for finite variance integrated supOU processes $X^*$,
	\item to explain how these results relate to intermittency.
\end{enumerate}

We deal with (i) in Section \ref{sec2}. We show that, depending on the conditions on the underlying supOU process, four different limiting processes may be obtained after suitable normalization, namely, fractional Brownian motion, stable L\'evy process, stable process with dependent increments defined below in \eqref{Zalphabeta} and Brownian motion. The nature of the limit will depend on the interplay between several components: whether there is a Gaussian component in the so-called characteristic triplet of the marginal distribution, how strong the dependence is and, somewhat surprisingly, it depends also on the growth of the L\'evy measure of the marginal distribution near the origin. In classical limit theorems, it is typically the tails of the marginal distribution that are important. Here, however, the behavior of the L\'evy measure near the origin may play an important role even though that behavior does not affect the tails of the marginal distribution. Note also that even though the integrated process has finite variance, it may happen that the limiting process has stable non-Gaussian marginal distribution and hence infinite variance. Examples are provided in Section \ref{sec2} illustrating the main results. The proofs of the limit theorems are given in Sections \ref{sec4} and \ref{s:proof-weak} and extend those of \cite{philippe2014contemporaneous} who consider certain discrete type superpositions of AR(1) processes.

In Section \ref{sec3} we investigate how the established limit theorems fit with the intermittency property. For each scenario of Section \ref{sec2}, we establish convergence of moments and derive the expressions for the scaling function for $q>0$. In general, the scaling function $\tau_{X^*}(q)$ will have the shape of a broken line indicating intermittency. The line starts at the origin but then changes slope at some higher value of $q$. This shows that in the intermittent case, the convergence of moments \eqref{limitformmom} does not hold beyond some critical value of $q$. Hence, it is possible to have both intermittency and limit theorems. This phenomenon, moreover, is not only restricted to the long-range dependent case and it, in fact, can happen even when the limit is Brownian motion, a process with independent increments. For further discussion see Section \ref{sec3}. 

In this sense the paper illustrates a new concept which can be named \textit{limit theorems with intermittency effect}. One possible scenario, but not the only one, could be as follows: assume that the aggregated process $X^*=\{X^*(t), \, t \geq 0\}$ can be decomposed as the sum of two independent processes $X_{1}^{*}$ and $X_{2}^{*}$, where for some $A_{T}$ the limit of the first normalized process is self-similar: $\{X_{1}^{*}(Tt)/A_{T}\}\rightarrow^{d} \{Z(t)\}$ as $T\rightarrow \infty$, while the second process does not influence the limit (for example $X_{2}^{*}(Tt)/A_{T}$ converges in probability to $0$ as $T\rightarrow \infty$, but $\E \left\vert X_{2}^{*}(Tt) / A_T\right\vert ^{q}\to \infty$ as $T\to \infty$, for some $q>2$). Then $\{X^{*}(Tt)/A_{T}\}\rightarrow^{d} \{Z(t)\}$ while $\E \left\vert X^{*}(Tt) / A_T\right\vert ^{q}\to \infty$ as $T\to \infty$.

The paper is organized as follows. In Section \ref{secPre} we define the supOU process and the underlying L\'evy basis. Limit theorems are stated in Section \ref{sec2} and moment behavior and intermittency are discussed in Section \ref{sec3}. Proofs are given in Sections \ref{sec4}, \ref{s:proof-weak} and \ref{s:proof-intermittency}.

\section{Preliminaries}\label{secPre}

\subsection{The supOU process}
A \textit{supOU} process is a strictly stationary process $X=\{X(t), \, t\in \R\}$ defined by the stochastic integral
\begin{equation}\label{supOU}
X(t)= \int_{\R_+} \int_{\R} e^{-\xi t + s} \mathbf{1}_{[0,\infty)}(\xi t -s) \Lambda(d\xi,ds).
\end{equation}
Here, $\Lambda$ is a homogeneous infinitely divisible random measure (\textit{L\'evy basis}) on $\R_+ \times \R$,  with cumulant function  
\begin{equation}\label{cumofLam}
C\left\{ \zeta \ddagger \Lambda(A)\right\} := \log \E e^{i \zeta \Lambda(A) } =  m(A) \kappa_{L}(\zeta) = \left( \pi \times Leb \right) (A) \kappa_{L}(\zeta),  
\end{equation}
for $A \in \mathcal{B} \left(\R_+ \times \R\right)$, thus involving the quantities $m$, $\pi$ and $\kappa_L$ which we now define. The \textit{control measure} $m=\pi \times Leb$ is the product of a probability measure $\pi$ on $\R_+$ and the Lebesgue measure on $\R$. The existence of the stochastic integral \eqref{supOU} in the sense of the paper \cite{rajput1989spectral} was proven by \cite{bn2001}. The probability measure $\pi$ ``randomizes'' the rate parameter $\xi$ and the Lebesgue measure is associated with the moving average variable $s$. Finally, $\kappa_{L}$ is the cumulant function $\kappa_{L} (\zeta)= \log \E e^{i \zeta L(1) }$ of some infinitely divisible random variable $L(1)$ with L\'evy-Khintchine triplet 
\begin{equation}\label{e:triplet}
(a,b,\mu_L),
\end{equation}
i.e.
\begin{equation}\label{kappacumfun}
\kappa_{L}(\zeta) = i\zeta a -\frac{\zeta ^{2}}{2} b  +\int_{\R}\left( e^{i\zeta x}-1-i\zeta x \mathbf{1}_{[-1,1]}(x)\right) \mu_L(dx).
\end{equation}
The quadruple 
\begin{equation}\label{quadruple}
(a,b,\mu_L,\pi)
\end{equation}
is referred to as the \textit{characteristic quadruple}.

\subsection{The marginal distribution}
The L\'evy process $L=\{L(t), \, t\geq 0\}$ associated with the triplet $(a,b,\mu_L)$ is termed the \textit{background driving L\'{e}vy process} and its law uniquely determines the one-dimensional marginal distribution of the process $X$ in \eqref{supOU} assuming $\E \log \left(1+ \left| L(1) \right| \right)< \infty$. 

We will consider self-decomposable distributions.\footnote{An infinitely divisible random variable $X(1)$ with characteristic function $\phi$ is self-decomposable if for every constant $c \in (0,1)$ there exists a characteristic function $\phi_c$ such that $\phi(\zeta)=\phi(c \zeta)\phi_c(\zeta)$, $\zeta\in \R$. Examples of self-decomposable distributions include Gamma, variance Gamma, inverse Gaussian, normal inverse Gaussian, Student and positive tempered stable distributions.

The definition of self-decomposability is related to the equation of the $AR(1)$ stationary process
\begin{equation*}
X_n=cX_{n-1}+\varepsilon_n, \quad c\in (0,1), \ n=1,2,\dots
\end{equation*}
Indeed, let $\phi(\zeta)$ denote the characteristic function of $X_n$, which does not depend on $n$ because of stationarity. Then the preceding equation becomes $\phi(\zeta)=\phi(c \zeta)\phi_c(\zeta)$, where $\phi_c$ is the characteristic function of $\varepsilon_n$ which must depend on $c$ to ensure stationarity. See \cite{wolfe1982continuous} for more details.} If $X$ is self-decomposable, then there corresponds a L\'evy process $L$ such that 
\begin{equation}\label{XtoLint}
X\overset{d}{=} \int_0^\infty e^{-s} dL(s).
\end{equation}

Hence, by appropriately choosing the background driving L\'{e}vy process $L$, one can obtain any self-decomposable distribution as a marginal distribution of $X$, and vice-versa. The cumulant functions of the background driving L\'evy process $L$ and the corresponding self-decomposable distribution $X$ are related by
\begin{align}
\kappa_{X}(\zeta )&=\int_{0}^{\infty }\kappa_{L}(e^{-s}\zeta) ds,\label{kappaXtoL} \\
\kappa_{L}(\zeta )&=\zeta \kappa_X'(\zeta ), \label{kappaLtoX}
\end{align}
where $\kappa_X'$ denotes the derivative of $\kappa_X$ (see e.g.~\cite{bn2001} or \cite{jurek2001remarks}).

\subsection{The dependence structure}
While the marginal distribution of supOU process is determined by $L$, the dependence structure is controlled by the probability measure $\pi$. Indeed, if $\E X(t)^2<\infty$, then the correlation function $r(\tau)$ of $X$ is the Laplace transform of $\pi$:
\begin{equation}\label{corrsupOULT}
r(\tau )=\int_{\R_+} e^{-\tau \xi }\pi (d\xi ), \quad \tau \geq 0.
\end{equation}
By a Tauberian argument, one easily obtains (see e.g.~\cite{fasen2007extremes}) that if for some $\alpha>0$ and some slowly varying function $\ell$
\begin{equation}\label{regvarofpi}
\pi \left( (0,x] \right) \sim \ell(x^{-1}) x^{\alpha}, \quad \text{ as } x \to 0,
\end{equation}
then the correlation function satisfies \eqref{regvarofACR} and, in particular, $\alpha \in (0,1)$ yields the long-range dependence. See \cite{bn2001}, \cite{barndorff2018ambit}, \cite{barndorff2005spectral}, \cite{barndorff2013levy}, \cite{barndorff2011multivariate}, \cite{GLST2017Arxiv} for more details about supOU processes.\\

\subsection{Notation}
Through the rest of the paper, $X$ will denote the supOU process defined in \eqref{supOU} with characteristic quadruple $(a,b,\mu_L,\pi)$ and $X^*$ will be the corresponding integrated process \eqref{integratedsupOU}. We assume $X$ has finite variance 
\begin{equation*}
\sigma^2=\Var X(t)<\infty.
\end{equation*}
For simplicity, we assume that the mean $\E X(t)=0$, otherwise one could add centering in the limit theorems.

We use the notation
\begin{equation*}
\kappa_Y(\zeta)=C\left\{ \zeta \ddagger Y\right\} = \log \E e^{i \zeta Y}
\end{equation*}
to denote the cumulant (generating) function of a random variable $Y$. For a stochastic process $Y=\{Y(t)\}$ we write $\kappa_Y(\zeta,t) = \kappa_{Y(t)}(\zeta)$, and by suppressing $t$ we mean $\kappa_Y(\zeta)=\kappa_Y(\zeta,1)$, that is the cumulant function of the random variable $Y(1)$. 

Note that if $X(1)$ has finite moments up to order $p$, then so does $L(1)$ (see \cite[Proposition 3.1]{fasen2007extremes}). Moreover, relation \eqref{kappaXtoL} implies that if $X(1)$ has zero mean, then the same is true for the background driving L\'{e}vy process $L(1)$. In this case, we can write the cumulant function of $L$ in the form (see e.g.~\cite[p.~39]{sato1999levy})
\begin{equation}\label{kappacumfun1}
\kappa_{L}(\zeta) = -\frac{\zeta ^{2}}{2} b + \int_{\R}\left( e^{i\zeta x}-1-i\zeta x\right) \mu_L(dx),
\end{equation}
where $b$ is the variance component. For such a representation we will use the notation $(0,b,\mu_L,\pi)_1$ for the characteristic quadruple. Note the presence of the index $1$ to indicate that the truncation function $\mathbf{1}_{[-1,1]}(x)$ has been replaced by the constant $1$ (see \cite[Section 8]{sato1999levy}). Note also that the variance of the Gaussian component $b$ and the L\'evy measure $\mu_L$ remain unchanged.

\section{Limit theorems}\label{sec2}
We start by assuming that $\alpha \in (0,1)$ in \eqref{regvarofpi}. This can be considered as the long-range dependence scenario. Indeed, $\alpha \in (0,1)$ implies that the correlation function is not integrable since by \eqref{corrsupOULT}
\begin{equation}\label{corrtopi}
\int_0^\infty r(\tau) d\tau = \int_0^\infty \int_0^\infty e^{-\tau \xi} d\tau \pi(d \xi) = \int_0^\infty \xi^{-1} \pi(d \xi) = \infty.
\end{equation}
To simplify the proofs of some of the results below, we will assume that $\pi$ has a density $p$ which is monotone on $(0,x')$ for some $x'>0$ so that \eqref{regvarofpi} implies
\begin{equation}\label{regvarofp}
p (x) \sim \alpha \ell(x^{-1}) x^{\alpha-1}, \quad \text{ as } x \to 0.
\end{equation}
Under long-range dependence different scenarios are possible depending on additional conditions. The following theorem shows that the limit is fractional Brownian motion if a Gaussian component $b\neq 0$ is present in the characteristic triplet of the marginal distribution (or equivalently in the characteristic triplet \eqref{e:triplet} of the background driving L\'{e}vy process). 

Recall that $\{\cdot\} \overset{fdd}{\to} \{\cdot\}$ appearing in particular in Theorems \ref{thm:FBMcase}-\ref{thm:BMcase}, denotes the convergence of finite dimensional distributions.

\begin{theorem}\label{thm:FBMcase}
Suppose that $\pi$ has a density $p$ satisfying \eqref{regvarofp} with $\alpha\in(0,1)$ and some slowly varying function $\ell$. If $b>0$ in \eqref{quadruple}, then as $T\to \infty$
\begin{equation*}
\left\{ \frac{1}{T^{1-\alpha/2 } \ell(T)^{1/2}} X^*(Tt) \right\} \overset{fdd}{\to} \left\{\widetilde{\sigma} B_H(t) \right\},
\end{equation*}
where $\{ B_H(t)\}$ is standard fractional Brownian motion with $H=1-\alpha/2$ and \begin{equation*}
\widetilde{\sigma}^2 = b \frac{\Gamma(1+\alpha)}{(2-\alpha)(1-\alpha)}.
\end{equation*}
\end{theorem}

The proof of this result and of the subsequent ones are given in Sections \ref{sec4}, \ref{s:proof-weak} and \ref{s:proof-intermittency}. 

The next scenario assumes that there is no Gaussian component namely $b=0$, so that the background driving process is a pure jump L\'evy process. In addition to the dependence parameter $\alpha$ in \eqref{regvarofp}, the limit in this setting will depend on the behavior of the L\'evy measure near the origin. Two limiting processes may arise in this setting both of which will have infinite variance stable marginals. Recall that the cumulant function of any $\gamma$-stable distributed random variable $Z$ such that $\E Z=0$ if $1<\gamma<2$, and $Z$ is symmetric if $\gamma=1$, can be written in the form (see e.g.~\cite[proof of Theorem 2.2.2]{ibragimov1971independent})
\begin{equation*}
C \left\{ \zeta \ddagger Z \right\} = - |\zeta|^{\gamma} \omega (\zeta; \gamma, c_1, c_2),
\end{equation*}
where
\begin{equation}\label{omega}
\begin{aligned}
&\omega (\zeta; \gamma, c_1, c_2) =\\
&\qquad \qquad \begin{cases}
\frac{\Gamma(2-\gamma)}{1-\gamma} \left( (c_1+c_2) \cos \left(\frac{\pi \gamma}{2}\right) - i (c_1-c_2) \sign (\zeta) \sin \left(\frac{\pi \gamma}{2}\right) \right), & \gamma\neq 1,\\
(c_1+c_2) \frac{\pi}{2}, & \gamma= 1,
\end{cases}
\end{aligned}
\end{equation}
with $c_1, c_2 \geq 0$ and $c_1=c_2$ if $\gamma=1$. By taking
\begin{equation*}
\sigma = \left( \frac{\Gamma(2-\gamma)}{1-\gamma} (c_1+c_2)  \cos \left(\frac{\pi \gamma}{2}\right) \right)^{1/\gamma}, \qquad \beta = \frac{c_1 - c_2}{c_1+c_2},
\end{equation*}
we may rewrite \eqref{omega} for $\gamma \neq 1$ as
\begin{equation*}
\omega (\zeta; \gamma, c_1, c_2) = \sigma^\gamma \left( 1- i \beta \sign (\zeta) \tan \left(\frac{\pi \gamma}{2}\right) \right), 
\end{equation*}
which is a more common parametrization (see e.g.~\cite[Definition 1.1.6]{samorodnitsky1994stable}).

For the first type of the limiting process we will assume that
\begin{equation}\label{muLint1+alpha}
\int_{|x|\leq 1} |x|^{1+\alpha} \mu_L(dx) < \infty.
\end{equation}
This is equivalent to $\beta_{BG}<1+\alpha$ where $\beta_{BG}$ is the \textit{Blumenthal-Getoor index} of the L\'evy measure $\mu_L$ which is defined as (\cite{blumenthal1961sample})\footnote{Clearly \eqref{muLint1+alpha} implies that $\beta_{BG}\leq 1+\alpha$, but it is possible to have $\beta_{BG} = 1+\alpha$ and $\int_{|x|\leq 1} |x|^{1+\alpha} \mu_L(dx) = \infty$ (for example if $\mu_L(dx)=|x|^{-1-\alpha} \log(|x|) dx$). However,  $\beta_{BG}< 1+\alpha$ does imply $\int_{|x|\leq 1} |x|^{1+\alpha} \mu_L(dx) < \infty$, hence the two are equivalent.}
\begin{equation}\label{BGindex}
\beta_{BG} = \inf \left\{\gamma \geq 0 : \int_{|x|\leq 1} |x|^\gamma \mu_L(dx) < \infty \right\}.
\end{equation}
Since $\mu_L$ is the L\'evy measure, we always have $\beta_{BG} \in [0,2]$. 

The normalization sequence in the following theorem involves de Bruijn conjugate of a slowly varying function. The de Bruijn conjugate of some slowly varying function $\ell$ is a unique slowly varying function $\ell^{\#}$ such that 
\begin{equation*}
\ell (x) \ell^{\#} (x \ell(x)) \to 1, \qquad \ell^{\#}(x) \ell( x \ell^{\#}(x) ) \to 1,
\end{equation*}
as $x\to \infty$ (see \cite[Theorem 1.5.13]{bingham1989regular}). In the setup of the following theorem, $\ell^{\#}$ is de Bruijn conjugate of $1/\ell\left(x^{1/(1+\alpha)}\right)$ with $\ell$ coming from \eqref{regvarofp}.

\begin{theorem}\label{thm:SLPcase}
Suppose that $\pi$ has a density $p$ satisfying \eqref{regvarofp} with $\alpha\in(0,1)$ and some slowly varying function $\ell$ and let $\beta_{BG}$ be defined by \eqref{BGindex}. If 
\begin{equation*}
b=0 \ \text{ and } \ \beta_{BG}<1+\alpha,
\end{equation*}
then as $T\to \infty$
\begin{equation*}
\left\{ \frac{1}{T^{1/(1+\alpha)} \ell^{\#}\left(T \right)^{1/(1+\alpha)}} X^*(Tt) \right\} \overset{fdd}{\to} \left\{L_{1 + \alpha} (t) \right\},
\end{equation*}
where $\ell^{\#}$ is de Bruijn conjugate of $1/\ell\left(x^{1/(1+\alpha)}\right)$ and  $\{L_{1+\alpha}\}$ is $(1+\alpha)$-stable L\'evy process such that 
\begin{equation*}
C \left\{ \zeta \ddagger L_{1 + \alpha} (1)  \right\} = - |\zeta|^{1+\alpha} \omega(\zeta; 1+ \alpha, c^-_\alpha, c^+_\alpha)
\end{equation*}
with $c^-_\alpha, c^+_\alpha$ given by
\begin{equation}\label{c+-alpha}
c^-_\alpha = \frac{\alpha}{1+\alpha} \int_{-\infty}^0 |y|^{1+\alpha} \mu_L(dy), \qquad c^+_\alpha = \frac{\alpha}{1+\alpha} \int_0^\infty y^{1+\alpha} \mu_L(dy).
\end{equation}
\end{theorem}

When 
\begin{equation*}
\int_{|x|\leq 1} |x|^{1+\alpha} \mu_L(dx) = \infty,
\end{equation*}
another stable process may arise in the limit. This time the limiting process will have dependent increments and it will depend on the rate of growth of the L\'evy measure near the origin. To quantify this rate of growth, we will assume a power law behavior of $\mu_L$ near origin. Let
\begin{align*}
M^+(x) &= \mu_L \left( [x, \infty) \right), \quad x>0\\
M^-(x) &= \mu_L \left( (-\infty, -x] \right), \quad x>0,
\end{align*}
denote the tails of $\mu_L$ and assume there exists $\beta>0$, $c^+, c^- \geq 0$, $c^++c^->0$ such that
\begin{equation}\label{LevyMCond}
M^+(x) \sim c^+ x^{-\beta} \ \text{ and } \ M^-(x) \sim c^- x^{-\beta} \ \text{  as } x \to 0.
\end{equation}
In particular, $\beta$ is the Blumenthal-Getoor index of $\mu_L$, $\beta=\beta_{BG}$. We will assume in the next theorem that $\beta>1+\alpha$. This implies that $\int_{|x|\leq 1} |x|^{1+\alpha} \mu_L(dx) \allowbreak = \infty$, hence this setting complements the one considered in Theorem \ref{thm:SLPcase}.

The property \eqref{LevyMCond} is stated in terms of the L\'evy measure $\mu_L$ of the background driving L\'{e}vy process $L$. We could, however, also state the condition in terms of the L\'evy measure $\mu_X$ of the corresponding self-decomposable distribution of $X$. Indeed, by \cite[Theorem 17.5]{sato1999levy} and Karamata's theorem \cite[Theorem 1.5.11]{bingham1989regular}, we have as $x\to 0$
\begin{equation}\label{muXcond}
\mu_X \left( [x, \infty) \right) = \int_0^\infty M^+ \left( e^s x \right) ds = \int_0^{1/x} M^+(s^{-1}) s^{-1} ds \sim \beta^{-1} M^+(x)
\end{equation}
and similarly $\mu_X \left( (-\infty, -x] \right) \sim \beta^{-1} M^-(x)$. Note that for \eqref{LevyMCond} the behavior of $\mu_L$ away from the origin is irrelevant.

\begin{theorem}\label{thm:Zcase}
Suppose that $\pi$ has a density $p$ satisfying \eqref{regvarofp} with $\alpha\in(0,1)$ and some slowly varying function $\ell$ and suppose \eqref{LevyMCond} holds with $\beta>0$. 
If 
\begin{equation*}
b=0 \ \text{ and } \ 1+\alpha<\beta<2,
\end{equation*}
then as $T\to \infty$
\begin{equation*}
\left\{ \frac{1}{T^{1-\alpha/\beta} \ell(T)^{1/\beta}} X^*(Tt) \right\} \overset{fdd}{\to} \left\{Z_{\alpha, \beta} (t) \right\},
\end{equation*}
where $\{Z_{\alpha, \beta}\}$ is a process with the stochastic integral representation
\begin{equation}\label{Zalphabeta}
Z_{\alpha, \beta}(t) = \int_{\R_+} \int_{\R} \left( \mathfrak{f}(\xi, t-s) - \mathfrak{f}(\xi, -s) \right) K(d\xi, ds),
\end{equation}
$\mathfrak{f}$ is given by
\begin{equation}\label{mff}
\mff(x,u) = \begin{cases}
1-e^{-xu}, & \text{ if } x>0 \text{ and } u>0,\\
0, & \text{ otherwise},
\end{cases}
\end{equation}
and $K$ is a $\beta$-stable L\'evy basis on $\R_+ \times \R$ with control measure \begin{equation*}
k(d\xi, ds)=\alpha \xi^{\alpha} d\xi ds,
\end{equation*}
such that 
\begin{equation*}
C \left\{ \zeta \ddagger K(A) \right\} = - |\zeta|^{\beta} \omega (\zeta; \beta, c^+, c^-) k(A).
\end{equation*}
\end{theorem}

The process $\{Z_{\alpha, \beta}\}$ defined in \eqref{Zalphabeta} was obtained by \cite{puplinskaite2010aggregation} in a similar limiting scheme. It is 
\begin{itemize}
\item $\beta$-stable,
\item $H=1-\alpha/\beta$ self-similar,
\item has stationary increments
\item and has continuous paths a.s.
\end{itemize}
This can be checked from the cumulant function of the finite dimensional distributions which is given by
\begin{equation}\label{e:Zalphabetafidis}
\begin{aligned}
C & \left\{ \zeta_1, \dots, \zeta_m \ddagger \left(Z_{\alpha, \beta}(t_1),\dots , Z_{\alpha, \beta}(t _m) \right) \right\}\\
&= - \int_0^\infty \int_{-\infty}^\infty \left| \sum_{j=1}^m \zeta_j  \left( \mathfrak{f}(\xi, t_j-s) - \mathfrak{f}( \xi, -s)  \right) \right|^\beta\\
&\qquad \times \omega \left( \sum_{j=1}^m \zeta_j  \left( \mathfrak{f}(\xi, t_j-s) - \mathfrak{f}( \xi, -s)  \right); \beta, c^+, c^- \right) \alpha \xi^{\alpha-\beta} ds d\xi.
\end{aligned}
\end{equation}
Indeed, consider $\{Z_{\alpha, \beta}(at)\}$ with $a>0$. To show self-similarity, namely that $\{Z_{\alpha, \beta}(at)\}\overset{d}{=}\{a^H Z_{\alpha, \beta}(t)\}$ with $H=1-\alpha/\beta$, use \eqref{e:Zalphabetafidis} and make the change of variables $\xi \to \xi/a$ and $s \to as$. This implies $\xi^{\alpha-\beta} ds d\xi \to a^{-(\alpha - \beta)} \xi^{\alpha-\beta} ds d\xi$, and hence $H=-(\alpha - \beta)/\beta=1-\alpha/\beta$. The continuity of the sample paths follows from the  Kolmogorov-Chentsov theorem (see e.g.~\cite[Theorem 2.8]{karatzas1991brownian}) since by self-similarity and stationarity of increments we have
\begin{equation*}
E \left|Z_{\alpha,\beta}(t)-Z_{\alpha,\beta}(s)\right|^{1+\alpha} \leq C \left|t-s\right|^{(1-\alpha/\beta)(1+\alpha)},
\end{equation*}
and $(1-\alpha/\beta)(1+\alpha)>1$.

It remains to consider the case when the correlation function is integrable which by \eqref{corrtopi} is equivalent to $\int_0^\infty \xi^{-1} \pi(d \xi)< \infty$. We can therefore think of this case as short-range dependence.

\begin{theorem}\label{thm:BMcase}
If $\int_0^\infty \xi^{-1} \pi(d \xi)< \infty$, then as $T\to \infty$
\begin{equation*}
\left\{ \frac{1}{T^{1/2}} X^*(Tt) \right\} \overset{fdd}{\to} \left\{\widetilde{\sigma} B(t) \right\},
\end{equation*}
where $\{ B(t)\}$ is standard Brownian motion and 
\begin{equation}\label{sigmatilBM}
\widetilde{\sigma}^2= 2 \sigma^2 \int_0^\infty \xi^{-1} \pi(d \xi), \quad \sigma^2=\Var X(1).
\end{equation}
\end{theorem}

Theorem \ref{thm:BMcase} covers, for example, the case of finite superpositions which are obtained by taking $\pi$ to be a probability measure with finite support. This special case was proved in \cite{GLST2016JSP} by using standard arguments. However, the assumption of Theorem \ref{thm:BMcase} also covers the case where $\pi$ satisfies \eqref{regvarofpi} with some $\alpha >1$. In this case the limit theorem coexists with intermittency as will be seen in the next section.\\

Based on the previous results, we can summarize the limiting behavior of the integrated finite variance supOU process. In the short-range dependent case, which is implied by $\alpha>1$ in \eqref{regvarofpi}, the limit is Brownian motion. When $\alpha<1$, the type of the limit depends on the L\'evy triplet of the marginal distribution. If a Gaussian component is present, the limit is fractional Brownian motion. If there is no Gaussian component, the limit may be a stable L\'evy process or the stable process \eqref{Zalphabeta} with dependent increments, depending on the behavior of the L\'evy measure $\mu_L$ in \eqref{kappacumfun} around the origin.

In order to summarize the results in a simplified manner, suppose \eqref{regvarofp} holds with some $\alpha>0$ and if $b=0$, suppose additionally that \eqref{LevyMCond} holds with some $0<\beta<2$. Let $\beta=2$ denote the case when the Gaussian component is present. Then the limits can be classified as shown in Figure \ref{fig:limits1} and Figure \ref{fig:limits2}.

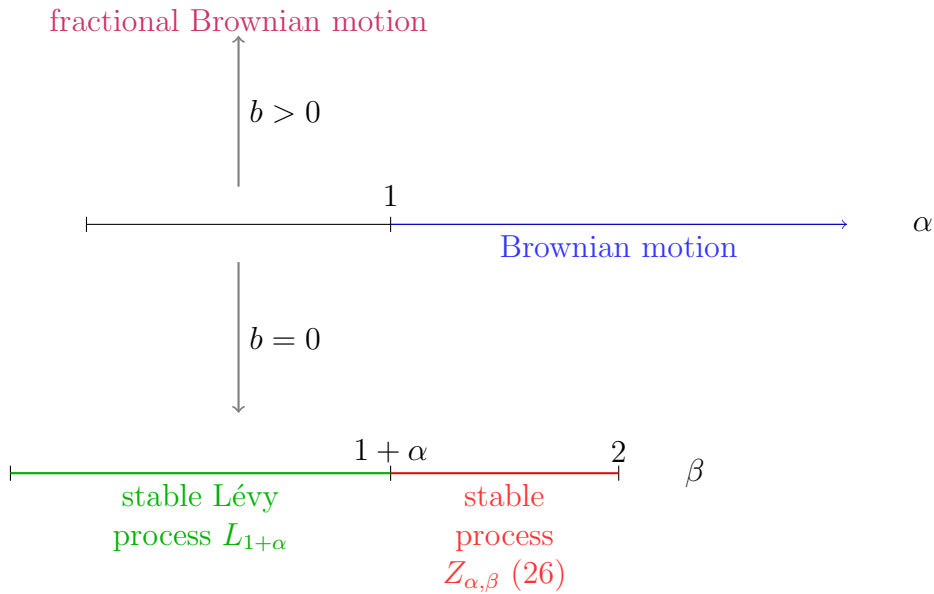
\begin{figure}[h!]
\centering
\begin{tikzpicture}
\draw (0,0) -- (10,0);
\node[] at (11,0) {$\alpha$};
\draw (0,-0.1) -- (0,0.1);
\draw (4,-0.1) -- (4,0.1);
\draw[->, blue, opacity=0.8](4,0) to (10,0);
\node[align=center, above] at (4,0.1) {$1$};
\node[align=center, below, blue, opacity=0.8] at (7,0) {Brownian motion};
\draw[->, thick,  gray](2, 0.5) to (2, 2.5);
\node[align=center, below, purple, opacity=0.8] at (2,3) {fractional Brownian motion};
\draw[->, thick, gray](2,-0.5) to (2,-2.5);
\node[align=center, right] at (2,1.5) {$b>0$};
\node[align=center, right] at (2,-1.5) {$b=0$};
\draw (-1,-3.3) -- (7,-3.3);
\draw (-1,-3.2) -- (-1,-3.4);
\draw (7,-3.2) -- (7,-3.4);
\draw (4,-3.2) -- (4,-3.4);
\node[align=center, above] at (7,-3.3) {$2$};
\node[align=center, above] at (4,-3.3) {$1+\alpha$};
\node[] at (8,-3.3) {$\beta$};
\node[align=center,  below, red, opacity=0.8, text width=0.15\textwidth] at (5.5,-3.3) {stable process $Z_{\alpha,\beta}$ \eqref{Zalphabeta}};
\node[align=center, below, black!30!green, text width=0.15\textwidth] at (1.5,-3.3) {stable L\'evy process $L_{1+\alpha}$};
\draw[thick, black!30!green] (-1,-3.3) -- (4,-3.3);
\draw[thick, red, opacity=0.8] (4,-3.3) -- (7,-3.3);
\end{tikzpicture}	
\caption{Classification of limits of $X^*$}
\label{fig:limits1}
\end{figure}
\FloatBarrier

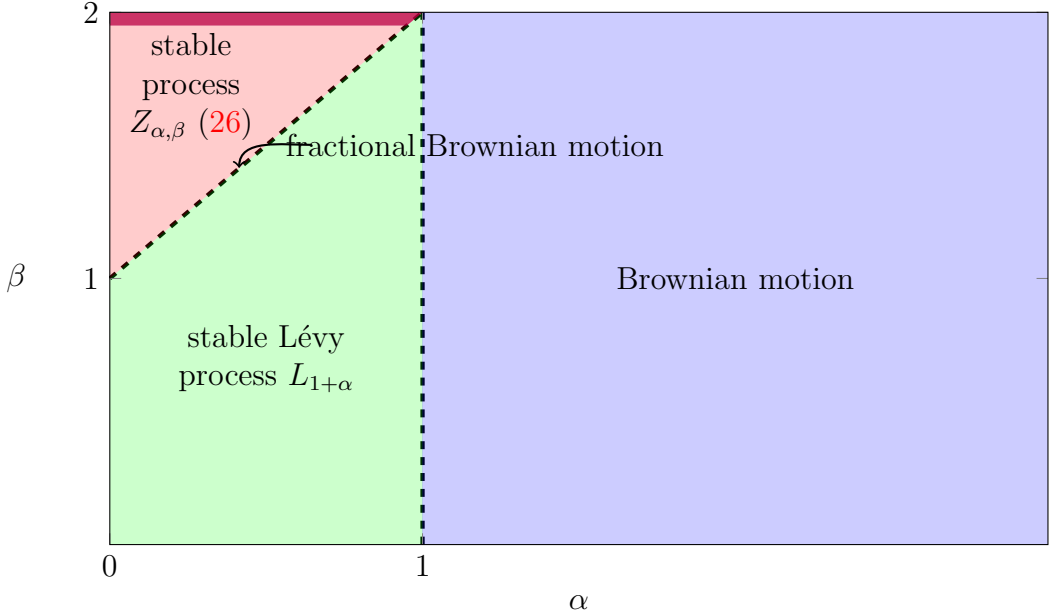
\begin{figure}[h!]
\centering
\begin{tikzpicture}
\begin{axis}[
width=0.87\textwidth,
height=0.54\textwidth,
xmin=0, xmax=3, ymin=0, ymax=2,
xlabel=$\alpha$, 
ylabel=$\beta$, ylabel style={rotate=-90}, 
xtick={0,1},
xticklabels={$0$,$1$},
ytick={1,2},
axis on top
]
\addplot[dashed, ultra thick] coordinates {(1,0) (1,2)};
\addplot[dashed, ultra thick] coordinates {(0,1) (1,2)};
\fill[blue, opacity=0.2] (axis cs:1,0) -- (axis cs:3,0) -- (axis cs:3,2) -- (axis cs:1,2);
\node[] at (axis cs:2,1) {Brownian motion};
\fill[green, opacity=0.2] (axis cs:0,0) -- (axis cs:1,0) -- (axis cs:1,2) -- (axis cs:0,1);
\node[align=center, text width=0.15\textwidth] at (axis cs:0.5,0.7) {stable L\'evy process $L_{1+\alpha}$};
\fill[red, opacity=0.2] (axis cs:0,1) -- (axis cs:0.95,1.95) -- (axis cs:0,1.95);
\node[align=center, text width=0.15\textwidth] at (axis cs:0.26,1.72) {stable process $Z_{\alpha,\beta}$ \eqref{Zalphabeta}};
\fill[purple, opacity=0.8] (axis cs:0,1.95) -- (axis cs:0.95,1.95) -- (axis cs:1,2.0) -- (axis cs:0,2.0);
\end{axis}
\node[] at (4.8,5.3) {fractional Brownian motion};
\node (source) at (2.8,5.3){};
\node (destination) at (1.7,4.85){};
\draw[->, thick](source) to [out=180,in=90] (destination);
\end{tikzpicture}	
\caption{Classification of limits of $X^*$}
\label{fig:limits2}
\end{figure}

Instead of using integrability of the correlation function, we may classify short-range and long-range dependence based on the dependence of increments of the limiting process. This way we could regard the case $1+\alpha>\beta$ as short-range dependence (Theorem \ref{thm:SLPcase}) and $1+\alpha<\beta$ as long-range dependence (Theorem \ref{thm:Zcase}). This implicitly includes the case $\beta=2$ when a Gaussian component is present which yields short-range dependence for $\alpha>1$ (Theorem \ref{thm:BMcase}) and long-range dependence for $\alpha<1$ (Theorem \ref{thm:FBMcase}).

Theorems \ref{thm:FBMcase}-\ref{thm:BMcase} establish convergence of finite dimensional distributions of normalized integrated process. The next theorem shows that in some cases the convergence may be extended to weak convergence in a suitable function space. Since we deal with the limits of the integrated process \eqref{integratedsupOU} which is continuous, we consider weak convergence in the space $C[0,1]$ of continuous function equipped with the uniform topology.

\begin{theorem}\label{thm:funconv}
The convergence in Theorems \ref{thm:FBMcase} and \ref{thm:Zcase} extends to the weak convergence in the space $C[0,1]$. The same is true for the convergence in Theorem \ref{thm:BMcase} if it additionally holds that $\E|X(t)|^4<\infty$ and $\int_0^\infty \xi^{-2} \pi (d\xi) <\infty$. 
\end{theorem}

In Theorem \ref{thm:SLPcase}, the limit is stable L\'evy motion. As noted by \cite{konstantopoulos1998macroscopic}, if a sequence of continuous processes converges, in the sense of finite dimensional distributions, to a limit which is discontinuous with positive probability, then the convergence cannot be extended to weak convergence in the space of c\`{a}dl\`{a}g functions $D[0,1]$ equipped with Skorokhod's $J_1$ topology \cite[p.~236]{konstantopoulos1998macroscopic}. Possibly the convergence holds in $D[0,1]$ equipped with the weaker $M_1$ topology or  the non-Skorohodian $S$ topology (see \cite{jakubowski1997non}). For such results in related models or limiting schemes see  \cite{doukhan2016discrete}, \cite{philippe2014contemporaneous}, \cite{resnick2000weak}.

\subsection{Examples}\label{examples:subsec}
In this subsection we list several examples of supOU processes and show how Theorems \ref{thm:FBMcase}--\ref{thm:BMcase} apply. In each example we will choose a background driving L\'evy process $L$ such that $L(1)$ is from some parametric class of infinitely divisible distributions. On the other hand, $\pi$ may be any absolutely continuous probability measure satisfying \eqref{regvarofp}. For example, $\pi $ can be Gamma distribution with density
\begin{equation*}
  f(x)= \frac{1}{\Gamma(\alpha)} x^{\alpha-1} e^{-x} \mathbf{1}_{(0,\infty)}(x),
\end{equation*}
where $\alpha>0$. Then
\begin{equation*}
  \pi ((0,x]) \sim \frac{1}{\Gamma(\alpha+1)} x^{\alpha}, \quad \text{ as } x \to 0.
\end{equation*}
Other examples of distributions satisfying \eqref{regvarofp} can be found in \cite{GLST2017Arxiv}.

For the limiting behavior of the corresponding integrated supOU process, the L\'evy-Khintchine triplet \eqref{e:triplet} of $L(1)$ will be important. In particular, for each case we will consider the value of $\beta$ defined in \eqref{LevyMCond} or the value of the Blumenthal-Getoor index \eqref{BGindex} of the L\'evy measure $\mu_L$. 

Note that one could also construct examples by choosing a marginal distribution of the supOU process instead of choosing the distribution of $L(1)$. Indeed, each distribution used in the examples below is self-decomposable, hence there exists a background driving L\'evy process generating a supOU process with such marginal distribution. By using the correspondence \eqref{muXcond}, one can easily check the conditions involving behavior of the L\'evy measure near origin.

\subsubsection{Compound Poisson background driving L\'evy process}
Suppose $L$ is a compound Poisson process with rate $\lambda>0$ and jump distribution $F$ having finite variance and zero mean. Suppose $X$ is a supOU process with the background driving L\'evy process $L$ and $\pi$ absolutely continuous probability measure satisfying \eqref{regvarofp}. The characteristic quadruple \eqref{quadruple} is then $(a,0,\mu_L,\pi)$ where
\begin{equation*}
a=\lambda \int_{|x|\leq 1} x F(dx), \qquad \mu_L(dx) = \lambda F(dx).
\end{equation*}
Since the L\'evy measure is finite, the Blumenthal-Getoor index \eqref{BGindex} is $0$. By Theorems \ref{thm:SLPcase} and \ref{thm:BMcase}, the limit of normalized integrated process is Brownian motion if $\alpha>1$ or stable L\'evy process with index $1+\alpha$ if $\alpha<1$.

\subsubsection{Normal inverse Gaussian background driving L\'evy process}
The normal inverse Gaussian distribution with shape parameter $A>0$, skewness parameter $|B|<A$, location parameter $C\in \R$ and scale parameter $D$ is given by the density
\begin{equation*}
f(x) = \frac{A}{\pi} e^{D \sqrt{A^2-B^2}} \left(1+\left(\frac{x-\mu}{D}\right)^2 \right)^{-1/2} K_1 \left( AD \sqrt{1+\left(\frac{x-\mu}{D}\right)^2 } \right) e^{B(x-C)}
\end{equation*}
for $x\in \R$, where $K_1$ is the modified Bessel function of third kind (see e.g.\cite{barndorff1997processes}, \cite{eberlein2004generalized}). The normal inverse Gaussian distributions are infinitely divisible and have all positive order moments finite. The density of the L\'evy measure is asymptotically equivalent to $\delta/\pi x^{-2}$ as $x\to 0$ (see \cite{eberlein2004generalized}), hence the Blumenthal-Getoor index \eqref{BGindex} is $\beta_{BG}=1$. Consider now a supOU process generated by the normal inverse Gaussian background driving L\'evy process. Since $\beta_{BG}<1+\alpha$, we conclude as in the compound Poisson case that the possible limits are Brownian motion if $\alpha>1$ or $(1+\alpha)$-stable L\'evy process if $\alpha<1$.

\subsubsection{Tempered stable background driving L\'evy process}
Let $L$ be a tempered stable L\'evy process, that is a L\'evy process with L\'evy-Khintchine triplet $(a,0,\mu_L)$ where $\mu_L$ is absolutely continuous with density given by
\begin{equation*}
g(y) = \frac{c^-}{|x|^{1+\beta}} e^{-\lambda^-|x|} \1_{(-\infty,0)}(x) + \frac{c^+}{x^{1+\beta}} e^{-\lambda^+ x} \1_{(0,\infty)}(x),
\end{equation*}
where $c^->0$, $c^+>0$, $\lambda^->0$, $\lambda^+>0$ and $\beta \in (0,2)$ (see e.g.~\cite[Section 4.5]{rama2003financial}). All moments of tempered stable distributions are finite and the L\'evy measure satisfies \eqref{LevyMCond} with $\beta \in (0,2)$. If $\alpha > 1$, then by Theorem \ref{thm:BMcase} the corresponding integrated supOU process converges to Brownian motion. In the case $\alpha<1$, the limit is $(1+\alpha)$-stable L\'evy process if $\beta < 1+\alpha$ (Theorem \ref{thm:SLPcase}), but if $\beta>1+\alpha$, then the limit is $\beta$-stable process \eqref{Zalphabeta} (Theorem \ref{thm:Zcase}).

\section{Moment behavior and intermittency}\label{sec3}

In this section we establish the asymptotic behavior of absolute moments of the integrated supOU process. More precisely, we investigate the scaling function of the integrated process
\begin{equation}\label{deftau}
\tau_{X^*}(q) = \lim_{t\to \infty} \frac{\log \E |X^*(t)|^q}{\log t}.
\end{equation}
We will assume throughout that the cumulant function $\kappa_{X}$ is analytic in the neighborhood of the origin. According to \cite[Theorem 7.2.1]{lukacs1970characteristic}, this is equivalent to the exponential decay of tails of the distribution of $X$. In particular, all moments are finite and the scaling function \eqref{deftau} will be well defined. Many infinitely divisible distributions satisfy this condition, for example, inverse Gaussian, normal inverse Gaussian, gamma, variance gamma, tempered stable (see \cite{GLST2017Arxiv} for details). It is worth noting that the same results could be obtained by assuming only that the moments exists up to some order, however, this would significantly complicate the exposition. Note also that the analyticity assumption does not affect the choice of $\pi$.

We noted in the introduction that integrated supOU processes may exhibit intermittency. As we will see, for the non-Gaussian supOU process with zero mean such that \eqref{regvarofpi} holds for some $\alpha>0$ we have that
\begin{equation}\label{tausupOUqqstar}
\tau_{X^*}(q) = q- \alpha, \quad \forall q \geq q^*,
\end{equation}
where $q^*$ is the smallest even integer greater than $2\alpha$ \cite[Theorem 7]{GLST2017Arxiv}. Hence, $q\mapsto \tau_{X^*}(q)/q$ is strictly increasing on $[q^*,\infty)$. On the other hand, there can be no normalizing sequence $A_T$ such that the normalized $q$-th moment $\E|X^*(T)/A_T|^q$ converges for every $q\geq q^*$. Indeed, if this normalized moment would converge $\E|X^*(T)/A_T|^q \to C_q$ for some $q$ as $T\to \infty$, then it follows that
\begin{equation*}
\log T \left( \frac{1}{q} \frac{\log \E|X^*(T)|^q}{\log T} -\frac{\log A_T}{ \log T} \right) \to \frac{1}{q} \log C_q, \quad T \to \infty,
\end{equation*}
which implies that $\log A_T/\log T \to \tau_{X^*}(q)/q$. Clearly, this is impossible to hold for more than one value of $q$, unless $\tau_{X^*}(q)/q$ is constant. Therefore we cannot have a limit theorem, the convergence of moments \eqref{limitformmom} and the unusual behavior of moments \eqref{tausupOUqqstar}. However, as the results of Section \ref{sec2} show, even when this unusual behavior of moments is present, it is still possible that a limit theorem holds after suitable normalization. What must fail to hold then is the convergence of moments \eqref{limitformmom}. Thus the convergence of moments \eqref{limitformmom} must not hold beyond some critical value of $q$. 

The purpose of this section is to provide a closer inspection of the behavior of moments in connection with the results of Section \ref{sec2}.

As in Section \ref{sec2}, we start with the case when $\alpha<1$ in \eqref{regvarofp}. First, we consider the setting of Theorem \ref{thm:FBMcase} where $\alpha \in (0,1)$ and where the limit is fractional Brownian motion.

\begin{theorem}\label{thm:momFBMcase}
Suppose that the assumptions of Theorem \ref{thm:FBMcase} hold, in particular $\alpha \in (0,1)$, and suppose $\mu_L \not\equiv 0$. Then
\begin{equation}\label{e:FBMtau1}
\tau_{X^*}(q) = \begin{cases}
\left( 1 - \frac{\alpha}{2}\right) q, & 0<q\leq 2,\\
q-\alpha, & q\geq 2.
\end{cases}
\end{equation}
If $\mu_L \equiv 0$, then $X^*$ is Gaussian and 
\begin{equation}\label{e:FBMtau0}
\tau_{X^*}(q) = \left( 1 - \frac{\alpha}{2}\right) q, \quad \forall q>0.
\end{equation}
\end{theorem}

It is interesting to note how intermittency appears in the setting of Theorem \ref{thm:momFBMcase}. Let $X^*_1$, $X^*_2$ denote the decomposition \eqref{e:thmFBMdecompostioninproof} of $X^*$ as in the proof of Theorem \ref{thm:FBMcase}, corresponding to Gaussian and pure jump part of the underlying L\'evy basis, respectively. With normalizing sequence $A_T=T^{1-\alpha/2 } \ell(T)^{1/2}$, we have for the Gaussian part $X_1^*$ and $t>0$
\begin{equation*}
A_T^{-1} X^*_1(Tt) \overset{d}{\to} \widetilde{\sigma} B_{1-\alpha/2} (t),
\end{equation*}
and (see the proof of Theorem \ref{thm:momFBMcase})
\begin{equation*}
\E \left| A_T^{-1} X^*_1(Tt) \right|^q \to \E \left| \widetilde{\sigma} B_{1-\alpha/2} (t) \right|^q, \quad \forall q>0,
\end{equation*}
where 
\begin{equation*}
\widetilde{\sigma}^2= b \frac{\Gamma(1+\alpha)}{(2-\alpha)(1-\alpha)}.
\end{equation*}
Consider now the L\'evy component $X_2^*$ for which we have
\begin{equation*}
A_T^{-1} X^*_2(Tt) \overset{P}{\to} 0,
\end{equation*}
by using the normalization $A_T$ as above. Borrowing the term from \cite{doukhan2016discrete}, we may call the process $\{A_T^{-1} X^*_2(Tt),\, T>0\}$ \textit{evanescent}. However, its moments are far from negligible in the limit since
\begin{equation*}
\E \left| A_T^{-1} X^*_2(Tt) \right|^q \to \infty, \quad \forall q>2.
\end{equation*}
We conclude that it is the component $X^*_2$ which is responsible for the unusual limiting behavior of moments. Note, however, that by Theorems \ref{thm:SLPcase} and \ref{thm:Zcase}, $X^*_2$ can still be normalized to obtain a limit with stable non-Gaussian distribution. The appropriate normalization is of an order lower than $A_T$ since
\begin{equation*}
\frac{1}{1+\alpha} < 1-\frac{\alpha}{\beta} < 1- \frac{\alpha}{2},
\end{equation*}
with the notation of Theorems \ref{thm:SLPcase} and \ref{thm:Zcase}. Note also that the variance of $A_T^{-1} X^*_2(Tt)$ converges to a finite constant. Indeed, taking the second derivative of $\kappa_{X_2^*(Tt)}(\zeta)$ in \eqref{e:kappa:x2normalized} below and letting $\zeta\to 0$ we get by using \eqref{e:norm1} and \eqref{slowlyvarying under int} that
\begin{align}
A_T^{-2} \E X_2^*(Tt)^2 &= 2 \E X_2(1)^2 A_T^{-2} \int_0^\infty \int_0^{Tt} \xi^{-1} \left(1- e^{-\xi (Tt-s)} \right) ds \pi(d\xi)\nonumber \\
&= 2 \E X_2(1)^2 A_T^{-2}  \int_0^{\infty} \left( 1 -e^{-w} \right) \int_{w/(Tt)}^{\infty} \xi^{-2} \pi(d\xi) dw \nonumber \\
&\to 2 \E X_2(1)^2 t^{2-\alpha} \frac{\Gamma(1+\alpha)}{(2-\alpha)(1-\alpha)}, \label{e:momFBM:part2}
\end{align}
as $T\to \infty$. In particular, $\{A_T^{-q} \left| X^*_2(Tt)\right|^q\}$ is not uniformly integrable for $q\geq 2$.

The following simple example replicates the type of behavior we encounter with $X^*(T)$ in Theorem \ref{thm:momFBMcase}. Suppose $\{Y_T, \, T\geq 1\}$ is a sequence of random variables such that
\begin{equation*}
Y_T=\begin{cases}
T^{1-\alpha/2}, \ \text{ with probability } 1-T^{-\alpha},\\
T, \ \text{ with probability } T^{-\alpha},
\end{cases}
\end{equation*}
where $\alpha \in (0,1)$. Then, since $\E Y_T^q = T^{(1-\alpha/2)q} (1-T^{-\alpha}) + T^{q-\alpha}$, we have that $\E Y_T^q \sim T^{(1-\alpha/2)q}$ for $q\leq 2$ and $\E Y_T^q \sim T^{q-\alpha}$ for $q>2$. With suitable normalization we have
\begin{equation*}
T^{-1+\alpha/2} Y_T=\begin{cases}
1, \ \text{ with probability } 1-T^{-\alpha},\\
T^{\alpha/2}, \ \text{ with probability } T^{-\alpha},
\end{cases}
\end{equation*}
hence $T^{-1+\alpha/2} Y_T \overset{d}{\to} 1$. However, for the moments it holds that
\begin{equation*}
\E \left( T^{-1+\alpha/2} Y_T \right)^q = 1-T^{-\alpha} + T^{\alpha\left(q/2-1\right)}\to \begin{cases}
1, & q< 2,\\
2, & q=2,\\
\infty, & q>2,
\end{cases}
\end{equation*}
because $Y_T$ exhibits increasingly large values albeit with decreasing probability. This type of behavior is intensively studied for random fields arising from stochastic partial differential equations (see e.g.~\cite{carmona1994parabolic,chong2017intermittency,khoshnevisan2014analysis} and the references therein).

The following two theorems describe the scaling function when there is no Gaussian component. The limiting processes, given in Theorems \ref{thm:SLPcase} and \ref{thm:Zcase}, have a stable distribution and are $H$-self-similar. The slopes of the scaling functions involve the self-similarity index $H$ of the limiting process, respectively $1/(1+\alpha)$ and $(1-\alpha/\beta)$.

\begin{theorem}\label{thm:momSLPcase}
Suppose that the assumptions of Theorem \ref{thm:SLPcase} hold, in particular $\alpha \in (0,1)$ and $\beta_{BG} < 1+\alpha$. Then
\begin{equation}\label{e:SLPtau}
\tau_{X^*}(q) = \begin{cases}
\frac{1}{1+\alpha} q, & 0<q\leq 1+\alpha,\\
q-\alpha, & q\geq 1+\alpha.
\end{cases}
\end{equation}
\end{theorem}

\begin{theorem}\label{thm:momZcase}
Suppose that the assumptions of Theorem \ref{thm:Zcase} hold, in particular $\alpha \in (0,1)$ and $1+\alpha<\beta<2$. Then
\begin{equation}\label{e:Ztau}
\tau_{X^*}(q) = \begin{cases}
\left(1-\frac{\alpha}{\beta} \right) q, & 0<q\leq \beta,\\
q-\alpha, & q\geq \beta.
\end{cases}
\end{equation}
\end{theorem}

We now turn to the short range dependent setting of Theorem \ref{thm:BMcase}. In this case the integrated supOU process need not be intermittent. For example, if $\pi$ is a measure with finite support, then the supOU process corresponds to a finite superposition of OU type processes which satisfies a strong mixing property. The limit is Brownian motion which is $H$-self-similar with $H=1/2$. From the results of \cite{yokoyama1980moment}, one may show uniform integrability which together with Theorem \ref{thm:BMcase} implies that $\tau_{X^*}(q)=Hq=q/2$ for every $q>0$ (see also \cite[Example 8]{GLST2017Arxiv}). However, when $\pi$ is regularly varying at zero, intermittency is present. The following theorem gives the form of the scaling function showing that the change-point between two linear parts is $2\alpha$.

\begin{theorem}\label{thm:momBMcase}
Suppose that $\pi$ satisfies \eqref{regvarofpi} with integer $\alpha>1$ and some slowly varying function $\ell$. If $\mu_L \not\equiv 0$, then
\begin{equation}\label{e:BMtau}
\tau_{X^*}(q) = \begin{cases}
\frac{1}{2} q, & 0<q\leq 2 \alpha,\\
q-\alpha, & q\geq 2 \alpha.
\end{cases}
\end{equation}
If $\mu_L \equiv 0$, then $X^*$ is Gaussian and 
\begin{equation*}
\tau_{X^*}(q) = \frac{1}{2} q, \quad \forall q>0.
\end{equation*}
\end{theorem}

Figure \ref{fig2} provides the plots of the scaling functions obtained in this section.

\begin{figure}
\centering
\begin{subfigure}[b]{0.45\textwidth}
\resizebox{1.2\textwidth}{!}{
\begin{tikzpicture}
\begin{axis}[
axis lines=middle,
xlabel=$q$, xlabel style={at=(current axis.right of origin), anchor=west},
ylabel=$\tau_{X^*}(q)$, ylabel style={at=(current axis.above origin), anchor=south},
xtick={0,2},
xticklabels={$0$,$2$},
domain=0:5,
ymajorticks=false
]
\addplot[thick,domain=0:2]{(1-0.25)*x} node [pos=0.75,left]{$\left(1-\frac{\alpha}{2}\right)q\ $};
\addplot[thick,domain=2:5]{x-0.5} node [pos=0.5,left]{$q-\alpha$};
\addplot[dashed] coordinates {(2,0) (2,1.5)};
\addplot[dashed] coordinates {(0,1.5) (2,1.5)};
\end{axis}
\end{tikzpicture}
}
\caption{Theorem \ref{thm:momFBMcase} (non-Gaussian case)}
\label{fig2a}
\end{subfigure}
\hfill
\begin{subfigure}[b]{0.45\textwidth}
\resizebox{1.2\textwidth}{!}{
\begin{tikzpicture}
\begin{axis}[
axis lines=middle,
xlabel=$q$, xlabel style={at=(current axis.right of origin), anchor=west},
ylabel=$\tau_{X^*}(q)$, ylabel style={at=(current axis.above origin), anchor=south},
xtick={0,1.5},
xticklabels={$0$,$1+\alpha$},
domain=0:5,
ymajorticks=false
]
\addplot[thick,domain=0:1.5]{(1/1.5)*x} node [pos=0.75,left]{$\frac{1}{1+\alpha}q\ $};
\addplot[thick,domain=1.5:5]{x-0.5} node [pos=0.5,left]{$q-\alpha$};
\addplot[dashed] coordinates {(1.5,0) (1.5,1)};
\addplot[dashed] coordinates {(0,1) (1.5,1)};
\end{axis}
\end{tikzpicture}
}
\caption{Theorem \ref{thm:momSLPcase}}
\label{fig2c}
\end{subfigure}
\begin{subfigure}[b]{0.45\textwidth}
\resizebox{1.2\textwidth}{!}{
\begin{tikzpicture}
\begin{axis}[
axis lines=middle,
xlabel=$q$, xlabel style={at=(current axis.right of origin), anchor=west},
ylabel=$\tau_{X^*}(q)$, ylabel style={at=(current axis.above origin), anchor=south},
xtick={0,1.7},
xticklabels={$0$,$\beta$},
domain=0:5,
ymajorticks=false
]
\addplot[thick,domain=0:1.7]{(1-0.5/1.7)*x} node [pos=0.83,left]{$\left(1-\frac{\alpha}{\beta}\right)q\, $};
\addplot[thick,domain=1.7:5]{x-0.5} node [pos=0.5,left]{$q-\alpha$};
\addplot[dashed] coordinates {(1.7,0) (1.7,1.2)};
\addplot[dashed] coordinates {(0,1.2) (1.7,1.2)};
\end{axis}
\end{tikzpicture}
}
\caption{Theorem \ref{thm:momZcase}}
\label{fig2b}
\end{subfigure}
\hfill
\begin{subfigure}[b]{0.45\textwidth}
\resizebox{1.2\textwidth}{!}{
\begin{tikzpicture}
\begin{axis}[
axis lines=middle,
xlabel=$q$, xlabel style={at=(current axis.right of origin), anchor=west},
ylabel=$\tau_{X^*}(q)$, ylabel style={at=(current axis.above origin), anchor=south},
xtick={0,3},
xticklabels={$0$,$2\alpha$},
domain=0:5,
ymajorticks=false
]
\addplot[thick,domain=0:3]{0.5*x} node [pos=0.5,above]{$\frac{1}{2}q$};
\addplot[thick,domain=3:5]{x-1.5} node [pos=0.5,left]{$q-\alpha$};
\addplot[dashed] coordinates {(3,0) (3,1.5)};
\addplot[dashed] coordinates {(0,1.5) (3,1.5)};
\end{axis}
\end{tikzpicture}
}
\caption{Theorem \ref{thm:momBMcase} (non-Gaussian case)}
\label{fig2d}
\end{subfigure}
\caption{Scaling functions obtained in Theorems \ref{thm:momFBMcase}-\ref{thm:momBMcase}. If $X^*$ is purely Gaussian, then the limit process is Gaussian and the scaling function is then a straight line $\tau_{X^*}(q)=Hq$, $q>0$, where $H\in (1/2,1)$ in the case of Theorem \ref{thm:momFBMcase} and $H=1/2$ in the case of Theorem \ref{thm:momBMcase}.}\label{fig2}
\end{figure}
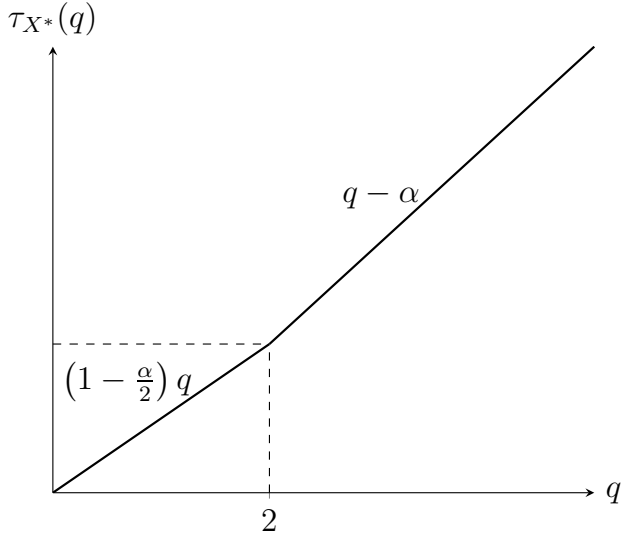
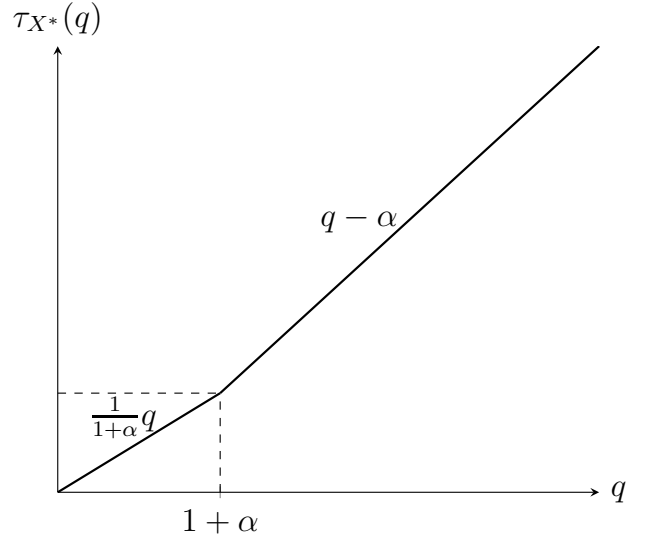
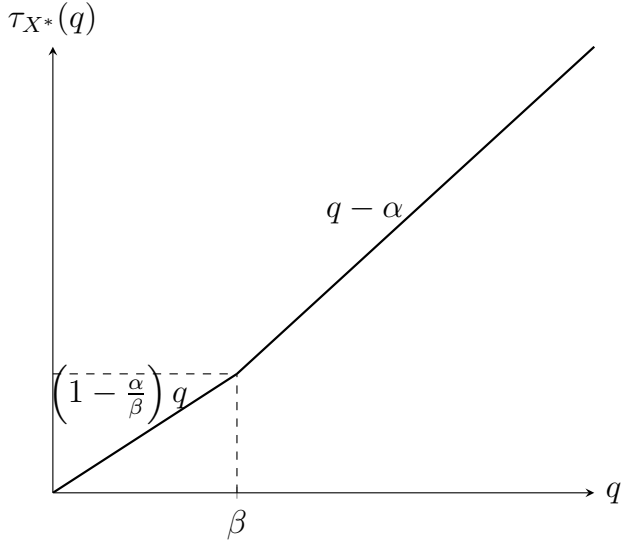
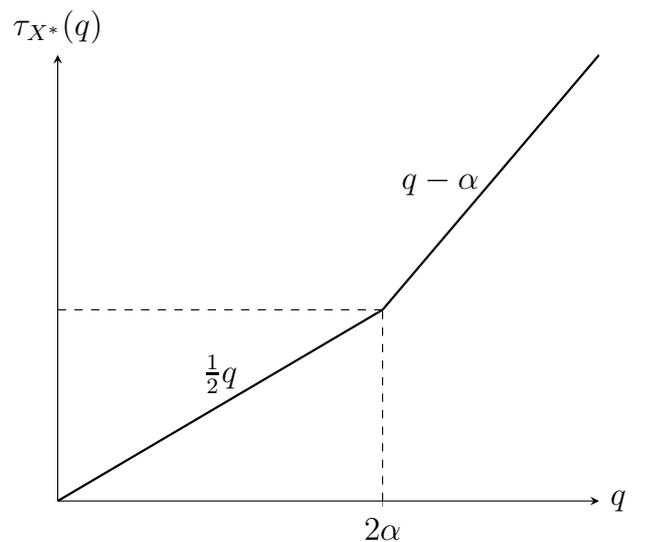

Theorem \ref{thm:momBMcase} assumes $\alpha$ in \eqref{regvarofpi} is an integer.  We conjecture,
however, the following:

\begin{conjecture}
Theorem \ref{thm:momBMcase} holds for any real $\alpha > 1$.
\end{conjecture}

In fact, a closer look at the proof of Theorem \ref{thm:momBMcase} reveals that we actually have
\begin{equation}  \label{e:conj}
\tau_{X^*}(q)=\begin{cases}
\frac{1}{2} q, & 0<q\leq q_*,\\
q-\alpha, & q\geq q^*.
\end{cases}
\end{equation}
where $q_*$ is the largest even integer less than or equal to $2\alpha$ and $q^*$ is the smallest even integer greater than $2\alpha$, as in \eqref{tausupOUqqstar} (see Figure \ref{figconjecture}). So, if $\alpha >1$ is integer, then $q_*=2\alpha$ and $\tau_{X^*}$ is a convex function \cite[Proposition 2.1]{GLST2016JSP} passing through three collinear points $(2\alpha,\alpha)$, $(q^*, q^*-\alpha)$ and, say, $(q^*+1, q^*+1-\alpha)$. Hence, $\tau_{X^*}$ must be linear on $[2\alpha,q^*]$ \cite[Lemma 3]{GLST2017Arxiv} and Theorem \ref{thm:momBMcase} follows. Relation \eqref{e:conj} shows that even if $\alpha>1$ is not an integer, $\tau_{X^*}$ is not a linear function.

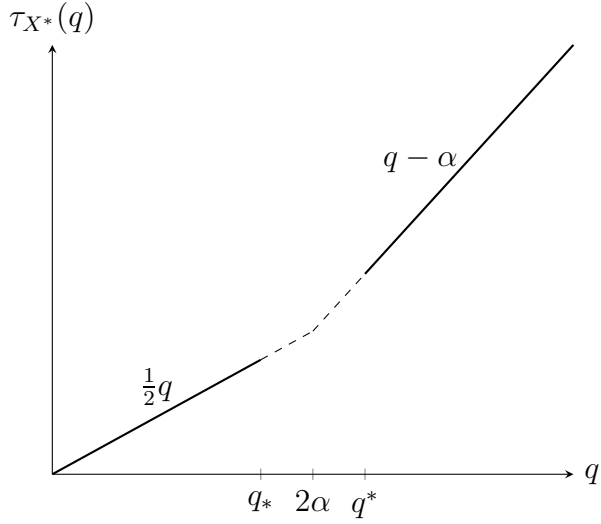
\begin{figure}[h!]
\centering
\begin{center}
	\begin{tikzpicture}
	\begin{axis}[
	axis lines=middle,
	xlabel=$q$, xlabel style={at=(current axis.right of origin), anchor=west},
	ylabel=$\tau_{X^*}(q)$, ylabel style={at=(current axis.above origin), anchor=south},
	xtick={0,4,5,6},
	xticklabels={$0$,$q_*$,$2\alpha$,$q^*$},
	domain=0:10,
	ymajorticks=false
	]
	\addplot[thick,domain=0:4]{0.5*x} node [pos=0.5,above]{$\frac{1}{2}q$};
	\addplot[thick,domain=6:10]{x-2.5} node [pos=0.5,left]{$q-\alpha$};
	\addplot[dashed,domain=4:5]{0.5*x};
	\addplot[dashed,domain=5:6]{x-2.5};
	\end{axis}
	\end{tikzpicture}
\end{center}	
\caption{Theorem \ref{thm:momBMcase} for non-integer $\alpha$. The solid part of the graph is given in \eqref{e:conj} and follows from the proof of Theorem \ref{thm:momBMcase}. The dashed part is a conjecture.}
\label{figconjecture}
\end{figure}

\section{Proofs related to convergence of finite dimensional distributions}\label{sec4}

Consider the supOU process $\{X(t), \, t \geq 0\}$ in \eqref{supOU} and the integrated process $\{X^*(t), \, t \geq 0\}$. The following lemma provides the joint cumulant function using the probability measure $\pi$, and either the cumulant function $\kappa_{L}$ in \eqref{kappacumfun} or the cumulant function $\kappa_X(\zeta)=\log \E e^{i\zeta X(1)}$ of $\{X(t), \, t \geq 0\}$.

\begin{lemma}
For $\zeta_1,\dots,\zeta_m \in \R$ and $t_1<\cdots<t_m$, the cumulant function of finite dimensional distributions of the normalized integrated process $X^*$ may be expressed as
\begin{align}
C & \left\{ \zeta_1, \dots, \zeta_m \ddagger \left(A_T^{-1} X^*( T t_1),\dots , A_T^{-1} X^*( T t _m) \right) \right\} \nonumber\\
\begin{split}
& = \int_0^\infty \int_{-\infty}^0 \kappa_L \left( A_T^{-1} \sum_{j=1}^m \zeta_j \xi^{-1} \left(e^{\xi s} - e^{-\xi (T t_j - s)} \right)  \right) ds \xi \pi(d\xi) \label{e:cumfidisnormLform1}\\
&\qquad  + \int_0^\infty \int_{0}^\infty \kappa_L \left( A_T^{-1} \sum_{j=1}^m \zeta_j \1_{[0,T t_j]}(s)  \xi^{-1} \left(1 - e^{-\xi (T t_j - s)} \right)  \right) ds \xi \pi(d\xi)\\
\end{split}\\
& = A_T^{-1} \sum_{i=1}^m \zeta_i \int_0^\infty \int_0^{T t_i} \kappa_X' \left( A_T^{-1} \sum_{j=1}^m \zeta_j \1_{[0,Tt_j]}(s) \xi^{-1} \left(1 - e^{-\xi (T t_j - s)} \right) \right) ds \pi(d\xi) \label{e:cumfidisnorm1}.
\end{align}
\end{lemma}

\begin{proof}
Following \cite{bn2001}, we can use $X(t)$ in \eqref{supOU} to define a generalized stochastic process (random linear functional) $\mathcal{X}$ by
\begin{equation*}
\mathcal{X} (f) = \int_0^\infty f(t) X(t) dt = \int_0^\infty \int_{-\infty}^\infty F_f(\xi,s) \Lambda(d\xi, ds),
\end{equation*}
where
\begin{equation*}
F_f(\xi,s) = e^s \int_0^\infty f(t) e^{-\xi t} \mathbf{1}_{[s/\xi,\infty)}(t) dt
\end{equation*}
and $f \in \mathcal{F}_{\Lambda}:= \left\{ f:[0,\infty) \to \R : F_f(\xi,s) \text{ is } \Lambda\text{-integrable} \right\}$. By \cite[Theorem 5.1]{bn2001} (note that the assumptions there are not necessary by \cite[Corollary 1]{jurek2001remarks}, for $f\in \mathcal{F}_{\Lambda}$ it holds that
\begin{align*}
C \left\{ f \ddagger \X \right\} &= C \left\{ 1 \ddagger \X(f) \right\} = \log \E \exp \left\{\int_0^\infty f(t) X(t) dt \right\}\\
&= \int_0^\infty \int_{-\infty}^\infty \kappa_L \left( \int_0^\infty f(t+s) e^{-\xi t} dt \right) ds \xi \pi(d\xi)\\
&= \int_0^\infty \int_{0}^\infty f(s) \kappa_X' \left( \int_0^\infty f(t+s) e^{-\xi t} dt \right) ds \pi(d\xi),
\end{align*}
by \eqref{kappaLtoX}, where we implicitly assume $f(t)=0$ for $t<0$. By letting $f(u)=\sum_{j=1}^m \zeta_j \1_{[0,t_j]}(u)$, we obtain two forms of the joint cumulant function of the integrated process $X^*$. One is
\begin{align*}
C & \left\{ \zeta_1, \dots, \zeta_m \ddagger \left(X^*(t_1),\dots , X^*(t_m) \right) \right\} \\
&\qquad = \int_0^\infty \int_{-\infty}^0 \kappa_L \left( \sum_{j=1}^m \zeta_j \int_{-s}^{t_j-s} e^{-\xi t} dt \right) ds \xi \pi(d\xi)\\
&\qquad \qquad + \int_0^\infty \int_{0}^\infty \kappa_L \left( \sum_{j=1}^m \zeta_j \1_{[0,t_j]}(s) \int_{0}^{t_j-s} e^{-\xi t} dt \right)  ds \xi \pi(d\xi)\\
&\qquad = \int_0^\infty \int_{-\infty}^0 \kappa_L \left( \sum_{j=1}^m \zeta_j \xi^{-1} \left(e^{\xi s} - e^{-\xi (t_j - s)} \right)  \right) ds \xi \pi(d\xi)\\
&\qquad \qquad + \int_0^\infty \int_{0}^\infty \kappa_L \left( \sum_{j=1}^m \zeta_j \1_{[0,t_j]}(s)  \xi^{-1} \left(1 - e^{-\xi (t_j - s)} \right)  \right) ds \xi \pi(d\xi).
\end{align*}
The other form involves the cumulant function  $\kappa_X$ of $\{X(t), \, t \geq 0\}$ and is obtained by using \eqref{kappaLtoX}:
\begin{align*}
C & \left\{ \zeta_1, \dots, \zeta_m \ddagger \left(X^*(t_1),\dots , X^*(t_m) \right) \right\} \\
&\qquad = \sum_{i=1}^m \zeta_i \int_0^\infty \int_0^{t_i} \kappa_X' \left( \sum_{j=1}^m \zeta_j \1_{[0,t_j]}(s) \int_{0}^{t_j-s} e^{-\xi t} dt \right) ds \pi(d\xi)\\
&\qquad = \sum_{i=1}^m \zeta_i \int_0^\infty \int_0^{t_i} \kappa_X' \left( \sum_{j=1}^m \zeta_j \1_{[0,t_j]}(s)  \xi^{-1} \left(1 - e^{-\xi (t_j - s)} \right)  \right) ds \pi(d\xi).
\end{align*}
From here one gets \eqref{e:cumfidisnormLform1} and \eqref{e:cumfidisnorm1}.
\end{proof}

\bigskip

\begin{proof}[Proof of Theorem \ref{thm:FBMcase}]
Suppose first that $\mu_L= 0$ so that $L$ is Brownian motion with $\kappa_{L}(\zeta)=i^2 \frac{b}{2} \zeta^2$. Then $X$ is a Gaussian supOU process and $\kappa_X(\zeta) = -\frac{b}{4} \zeta^2$ by \eqref{kappaXtoL}. From \eqref{e:cumfidisnorm1} we have
\begin{equation}\label{gaussianmarginalsfidis}
C  \left\{ \zeta_1, \dots, \zeta_m \ddagger \left(A_T^{-1} X^*( T t_1),\dots , A_T^{-1} X^*( T t _m) \right) \right\} = - \frac{b}{2} \sum_{i=1}^m  \sum_{j=1}^m  \zeta_i \zeta_j R_T (t_i, t_j),
\end{equation}
where $A_T=T^{1-\alpha/2 } \ell(T)^{1/2}$ and
\begin{equation}\label{RTtitj}
R_T(t_i, t_j) = A_T^{-2} \int_0^\infty \int_0^{T t_i \wedge T t_j}\left( 1 -e^{-\xi (T t_j - s)} \right) ds  \xi^{-1} \pi(d\xi).
\end{equation}
By a change of variables we have
\begin{align}
R_T(t_i, t_j) & = A_T^{-2} \int_0^\infty \int_0^{T t_i \wedge T t_j}\left( 1 -e^{-\xi (T t_j - s)} \right) ds  \xi^{-1} \pi(d\xi)\label{e:norm1}\\
&= A_T^{-2} \int_0^\infty \int^{\xi Tt_j}_{\xi (T_j -T t_i \wedge T t_j) }\left( 1 -e^{-w} \right) dw  \xi^{-2} \pi(d\xi)\nonumber\\
&= A_T^{-2}\int_0^\infty \int^{\xi Tt_j}_{0}\left( 1 -e^{-w} \right) dw  \xi^{-2} \pi(d\xi)\nonumber\\
&\qquad \qquad -  A_T^{-2} \int_0^\infty \int_0^{\xi (Tt_j -T t_i \wedge T t_j) }\left( 1 -e^{-w} \right) dw  \xi^{-2} \pi(d\xi)\nonumber\\
&= A_T^{-2} \int_0^{\infty} \left( 1 -e^{-w} \right) \int_{w/(Tt_j)}^{\infty} \xi^{-2} \pi(d\xi) dw\label{RTpart1proof}\\
&\qquad \qquad - A_T^{-2} \int_0^{\infty} \left( 1 -e^{-w} \right)  \int_{w/ (Tt_j -T t_i \wedge T t_j) }^{\infty} \xi^{-2} \pi(d\xi) dw.\label{RTpart2proof}
\end{align}
Here we implicitly assume the second term vanishes if $t_i\wedge t_j=t_j$.

We next show that
\begin{equation}\label{slowlyvarying under int}
\int_0^{\infty} \left( 1 -e^{-w} \right) \int_{w/t}^{\infty} \xi^{-2} \pi(d\xi) dw \sim \frac{\Gamma(1+\alpha)}{(2-\alpha)(1-\alpha)} \ell(t)  t^{2-\alpha} 
\end{equation}
as $t\to \infty$. Indeed, by \eqref{regvarofp}, $\xi \mapsto p(\xi^{-1})$ is $(1-\alpha)$-regularly varying at infinity and by the change of variables $u=1/\xi$ and by using Karamata's theorem \cite[Theorem 1.5.11]{bingham1989regular} we have as $t\to \infty$
\begin{align*}
\int_{w/t}^{\infty} \xi^{-2} \pi(d\xi) &= \int_{w/t}^{\infty} \xi^{-2} p(\xi) d\xi = \int_{0}^{t/w} p(u^{-1}) du \sim \frac{1}{2-\alpha} \frac{t}{w} p(w/t)\\
&\sim \frac{\alpha}{2-\alpha} \ell(t/w) \left(\frac{t}{w} \right)^{2-\alpha}.
\end{align*}
Hence, the integral $\int_{1/t}^{\infty} \xi^{-2} \pi(d\xi)$ is regularly varying function at infinity in $t$ and it can be written in the form
\begin{equation}\label{proof:thm:intermittency-integrated:2}
\int_{1/t}^{\infty} \xi^{-2} \pi(d\xi) = \frac{\alpha}{2-\alpha} \ell_1(t) t^{2-\alpha},
\end{equation}
with $\ell_1$ slowly varying at infinity such that $\ell_1(t) \sim \ell(t)$ as $t \to \infty$. Consequently, we have
\begin{align*}
\int_0^{\infty} \left( 1 -e^{-w} \right) \int_{w/t}^{\infty} \xi^{-2} \pi(d\xi) dw &= \frac{\alpha}{2-\alpha} t^{2-\alpha} \int_0^\infty \ell_1(t/w)\left(1- e^{-w} \right) w^{\alpha-2} dw\\
&= \frac{\alpha}{2-\alpha} t^{2-\alpha} \int_0^\infty \ell_1(t z) \left(1- e^{-1/z} \right) z^{-\alpha} dz.
\end{align*}
It remains to show that the integral on the right varies slowly in $t$. The function $g(z):=(1-e^{-1/z}) z^{-\alpha}$ is regularly varying at infinity with index $-\alpha-1$ and regularly varying at zero with index $-\alpha$. Hence, we can choose $0<\delta<1-\alpha$ so that
\begin{equation}\label{proof:thm:intermittency-integrated:delta0}
\int_0^1 z^{-\delta} g(z) dz < \infty.
\end{equation}
From \eqref{proof:thm:intermittency-integrated:2} we have that
\begin{equation*}
\ell_1(t) = \frac{2-\alpha}{\alpha} t^{\alpha-2} \int_{1/t}^{\infty} \xi^{-2} p(\xi) d\xi  \leq \frac{2-\alpha}{\alpha} t^{\alpha},
\end{equation*}
since $p$ is the probability density. Therefore $t^{\delta}\ell_1(t)$ is locally bounded on $[0,\infty)$. By applying \cite[Proposition 4.1.2(a)]{bingham1989regular} it follows that
\begin{equation*}
\int_0^1 \ell_1(t z) g(z) dz  \sim \ell_1(t) \int_0^1 g(z) dz, \quad \text{ as } t \to \infty.
\end{equation*}
On the other hand, for $0<\delta<\alpha$
\begin{equation*}
\int_1^\infty z^{\delta} g(z) dz < \infty
\end{equation*}
and by application of \cite[Proposition 4.1.2(b)]{bingham1989regular} we obtain
\begin{equation*}
\int_1^\infty \ell_1(t z) g(z) dz  \sim \ell_1(t) \int_1^\infty g(z) dz, \quad \text{ as } t \to \infty.
\end{equation*}
Integrating by parts and using the properties of the Gamma function we have
\begin{equation*}
\int_0^\infty g(z) dz = \frac{1}{1-\alpha} \Gamma (1+\alpha) = \frac{\Gamma(\alpha)}{\alpha (1-\alpha)}.
\end{equation*}
This completes the proof of \eqref{slowlyvarying under int}.

Returning back to \eqref{RTpart1proof} and \eqref{RTpart2proof}, we obtain as $T\to \infty$ that
\begin{align*}
& R_T(t_i, t_j) \sim  A_T^{-2}  (Tt_j)^{2-\alpha} \ell_1(Tt_j) \frac{\Gamma(1+\alpha)}{(2-\alpha)(1-\alpha)} \\
&\qquad - A_T^{-2} (Tt_j -T t_i \wedge T t_j)^{2-\alpha} \ell_1(Tt_j -T t_i \wedge T t_j)  \frac{\Gamma(1+\alpha)}{(2-\alpha)(1-\alpha)} \\
&=\left( t_j^{2-\alpha} \frac{\ell_1(Tt_j)}{\ell(T)} - (t_j - t_i\wedge t_j)^{2-\alpha} \frac{\ell_1(T(t_j - t_i\wedge t_j))}{\ell(T)} \right) \frac{\Gamma(1+\alpha)}{(2-\alpha)(1-\alpha)} .
\end{align*}
By using the fact that $\ell_1(t) \sim \ell(t)$ as $t\to\infty$, it follows that
\begin{equation}\label{e:thmFBMRTconverges}
\lim_{T\to \infty} R_T(t_i, t_j) = \left( t_j^{2-\alpha} - (t_j - t_i\wedge t_j)^{2-\alpha} \right) \frac{\Gamma(1+\alpha)}{(2-\alpha)(1-\alpha)}.
\end{equation}
Since
\begin{equation*}
\lim_{T\to \infty} R_T(t_i, t_j) + \lim_{T\to \infty} R_T(t_j, t_i)= \left( t_j^{2-\alpha} + t_i^{2-\alpha} - |t_j - t_i|^{2-\alpha} \right) \frac{\Gamma(1+\alpha)}{(2-\alpha)(1-\alpha)},
\end{equation*}
we can rewrite \eqref{gaussianmarginalsfidis} after taking the limit $T\to\infty$ in the form
\begin{equation*}
- b  \frac{\Gamma(1+\alpha)}{(2-\alpha)(1-\alpha)} \sum_{i=1}^m  \sum_{j=1}^m  \zeta_i \zeta_j \frac{1}{2} \left( t_j^{2-\alpha} + t_i^{2-\alpha} - |t_j - t_i|^{2-\alpha} \right),
\end{equation*}
which gives the finite dimensional distributions of fractional Brownian motion. This proves the statement for the Gaussian supOU. Note that instead of the direct proof one could also use general results for Gaussian processes, e.g.~\cite[Lemma 5.1]{Taqqu1975}, as in \cite[Example 9]{GLST2017Arxiv}.\\


Assume now that $\mu_L \not\equiv 0$. Then we can make a decomposition of the L\'evy basis into $\Lambda_1$ with characteristic quadruple $(0,b,0,\pi)_1$ and $\Lambda_2$ with characteristic quadruple $(0,0,\mu_L,\pi)_1$. Consequently, we can represent $X(t)$ as
\begin{equation}\label{e:thmFBMdecompostioninproof}
X(t) = \int_{0}^\infty \int_{-\infty }^{\xi t}e^{-\xi t + s} \Lambda_1(d\xi,ds) + \int_{0}^\infty \int_{-\infty }^{\xi t}e^{-\xi t + s} \Lambda_2(d\xi,ds) =: X_1(t) + X_2(t)
\end{equation}
with $X_1$ and $X_2$ independent. Let $X^*_1(t)$, $X^*_2(t)$ be the corresponding integrated processes. Since $X^*_1(t)$ is Gaussian, the preceding argument applies to show convergence to fractional Brownian motion. It remains to prove that $A_T^{-1} X^*_2(Tt) \overset{P}{\to} 0$ as $T\to \infty$. We shall do so by showing that its cumulant function tends to $0$ as $T\to \infty$.

Let $\kappa_{L,2}$ denote the cumulant function corresponding to $\Lambda_2$, i.e.
\begin{equation}\label{e:cum2}
\kappa_{L,2} (\zeta) = \int_{\R}\left( e^{i\zeta x}-1-i\zeta x \right) \mu_L(dx),
\end{equation}
and suppose $\kappa_{X,2}$ is the cumulant function of the corresponding selfdecomposable distribution (see \eqref{kappaXtoL} and \eqref{kappaLtoX}). By \eqref{regvarofp}, we can write $p$ in the form $p(x)=\alpha \widetilde{\ell}(x^{-1}) x^{\alpha-1}$ with $\widetilde{\ell}$ slowly varying at infinity such that $\widetilde{\ell}(t) \sim \ell(t)$ as $t \to \infty$. From \eqref{e:cumfidisnorm1} we have by a change of variables
\begin{equation}\label{e:kappa:x2normalized}
\begin{aligned}
& C \left\{ \zeta \ddagger T^{-1 + \alpha/2 } \ell(T)^{-1/2} X^*_2(Tt) \right\} =\\
& = T^{-1+\alpha/2} \ell(T)^{-1/2} \zeta \\
&\hspace{3em} \times \int_0^\infty \int_0^{Tt} \kappa_{X,2}' \left( T^{-1+\alpha/2} \ell(T)^{-1/2} \xi^{-1} \left( 1- e^{-\xi (Tt-s)} \right) \zeta\right) ds \pi(d\xi)\\
& = T^{\alpha/2} \ell(T)^{-1/2} \zeta \\
&\hspace{3em} \times \int_0^\infty \int_0^{t} \kappa_{X,2}' \left( T^{-1+\alpha/2} \ell(T)^{-1/2} \xi^{-1} \left( 1- e^{-\xi T s} \right) \zeta\right) ds \pi(d\xi)\\
& = T^{\alpha/2} \ell(T)^{-1/2} \zeta \\
&\hspace{3em} \times \int_0^\infty \int_0^{t} \kappa_{X,2}' \left( T^{\alpha/2} \ell(T)^{-1/2} \xi^{-1} \left( 1- e^{-\xi s} \right) \zeta\right) ds \pi(T^{-1} d\xi).
\end{aligned}
\end{equation}
Since $\pi(dx)=\alpha \widetilde{\ell}(x^{-1}) x^{\alpha-1} dx$, the last equation becomes
\begin{align}
&T^{\alpha/2} \ell(T)^{-1/2} \zeta \nonumber \\
&\hspace{3em} \times \int_0^\infty \int_0^{t} \kappa_{X,2}' \left( T^{\alpha/2} \ell(T)^{-1/2} \xi^{-1} \left( 1- e^{-\xi s} \right) \zeta\right) \alpha \widetilde{\ell}(T \xi^{-1}) \xi^{\alpha-1} T^{-\alpha} ds d\xi\nonumber\\
& = \alpha \zeta^2 \int_0^\infty \int_0^{t} k \left( T^{\alpha/2} \ell(T)^{-1/2} \xi^{-1} \left( 1- e^{-\xi s} \right) \zeta\right) \frac{\widetilde{\ell}(T \xi^{-1})}{\ell(T)} \left( 1- e^{-\xi s} \right) \xi^{\alpha-2} ds d\xi\label{FBMcaseproof:fununderint}.
\end{align}
We now focus on the function $k(\zeta)$. Since by \eqref{kappaLtoX} $\kappa_{L,2}(\zeta)=\zeta \kappa_{X,2}'(\zeta)$, we have
\begin{equation*}
k(\zeta) = \frac{\kappa_{X,2}'(\zeta)}{\zeta} = \frac{\kappa_{L,2}(\zeta)}{\zeta^2}.
\end{equation*}
By \eqref{e:cum2}, $\kappa_{L,2} (\zeta) = \int_{\R}\left( e^{i\zeta x}-1-i\zeta x \right) \mu_L(dx)$. Since $\left|e^{i\zeta x} - 1-i\zeta x\right| \leq \frac{1}{2} \zeta^2 x^2$, we get
\begin{equation}\label{e:bd}
\frac{\left| \kappa_{L,2}(\zeta) \right|}{\zeta^2} \leq \frac{1}{2} \int_{\R} x^2 \mu_L(dx) \leq  C
\end{equation}
for any $\zeta \in \R$. Hence, by the dominated convergence theorem $|\kappa_{L,2} (\zeta) | / \zeta^2 \to 0$ as $\zeta \to \infty$ (see also \cite[Eq.~(39)]{philippe2014contemporaneous}). We conclude that $k$ is bounded function such that $|k(\zeta)|\to 0$ as $\zeta\to \infty$.

Let $h_T(\xi,s)$ denote the function under the integral in \eqref{FBMcaseproof:fununderint}. Since $h_T(\xi,s) \to 0$ as $T\to \infty$, it remains to show that the dominated convergence theorem is applicable. Take $0<\delta<\min \left\{ \alpha,1-\alpha\right\}$. By Potter's bounds \cite[Theorem 1.5.6]{bingham1989regular}, there is $C_1$ such that $\widetilde{\ell}(T \xi^{-1})/\ell(T) \leq C_1 \max \left\{ \xi^{-\delta}, \xi^{\delta}\right\}$ and hence
\begin{equation*}
|h_T(\xi,s)| \leq C C_1 \max \left\{ \xi^{-\delta}, \xi^{\delta}\right\} \xi^{\alpha-2} \left( 1- e^{-\xi s} \right),
\end{equation*}
which is integrable. Indeed,
\begin{align*}
&\int_0^\infty \int_0^{t} \max \left\{ \xi^{-\delta}, \xi^{\delta}\right\} \xi^{\alpha-2} \left( 1- e^{-\xi s} \right) ds d\xi\\
&\qquad = \int_0^1 \xi^{\alpha-3-\delta} \left( e^{-\xi t} - 1 + \xi t \right)  d\xi + \int_1^\infty \xi^{\alpha-3 + \delta} \left( e^{-\xi t} - 1 + \xi t \right) d\xi\\
&\qquad \leq  t^2 \int_0^1 \xi^{\alpha-1-\delta} d\xi + t \int_1^\infty \xi^{\alpha-2 + \delta} d\xi < \infty.
\end{align*}
Hence the cumulant function of $A_T^{-1}X_2^*(Tt)$ tends to $0$ as $T\to \infty$. Therefore $A_T^{-1}X_2^*(Tt)$ tends to $0$ in distribution and hence in probability. 
\end{proof}

\begin{proof}[Proof of Theorem \ref{thm:SLPcase}]
Let $0=t_0<t_1<\cdots<t_m$, $\zeta_1,\dots,\zeta_m \in \R$ and $A_T=\allowbreak T^{1/(1+\alpha)} \allowbreak \ell^{\#}(T)^{1/(1+\alpha)} $. Note that the de Bruijn conjugate $\ell^{\#}$ exists by \cite[Theorem 1.5.13]{bingham1989regular} and satisfies
\begin{equation}\label{e:thmSPL2dBcon}
\frac{\ell^{\#} \left(T\right)}{\ell\left( \left( T\ell^{\#}\left(T \right) \right)^{1/(1+\alpha)} \right) } \sim 1, \text{  as } T\to \infty.
\end{equation}
It will be enough to prove that
\begin{equation}\label{e:thm:SLP1}
\sum_{i=1}^m \zeta_i A_T^{-1} \left( X^*(Tt_{i}) - X^*(Tt_{i-1}) \right) \overset{d}{\to} \sum_{i=1}^m \zeta_i \left( L_{1+\alpha} (t_i) - L_{1+\alpha} (t_{i-1}) \right).
\end{equation}
By using \eqref{supOU} we have that
\begin{equation}\label{e:thmSLP:decomposition}
\begin{aligned}
X^*(Tt_{i})& - X^*(Tt_{i-1}) = \int_{T t_{i-1}}^{T t_i} \int_0^\infty \int_{-\infty}^{\xi u} e^{-\xi u + s} \Lambda(d\xi,ds) du  \\
&=\int_0^\infty \int_{-\infty}^{\xi T t_{i-1}} \int_{T t_{i-1}}^{T t_i} e^{-\xi u + s} du \Lambda(d\xi,ds)\\
& \qquad \qquad +  \int_0^\infty \int_{\xi T t_{i-1}}^{\xi T t_i} \int_{s/\xi}^{T t_i} e^{-\xi u + s} du \Lambda(d\xi,ds)\\
&=: \Delta X^*_{(1)}(Tt_{i}) + \Delta X^*_{(2)}(Tt_{i})
\end{aligned}
\end{equation}
with $\Delta X^*_{(1)}(Tt_{i})$ and $\Delta X^*_{(2)}(Tt_{i})$ independent. Moreover, $\Delta X^*_{(2)}(Tt_{i})$, $i=1,\dots,m$ are independent, hence, to prove \eqref{e:thm:SLP1}, it will be enough to prove that
\begin{align}
A_T^{-1} \Delta X^*_{(1)}(Tt_{i}) &\overset{d}{\to} 0,\label{e:thm:SLPtoprove1}\\
A_T^{-1} \Delta X^*_{(2)}(Tt_{i}) &\overset{d}{\to}  L_{1+\alpha} (t_i) - L_{1+\alpha} (t_{i-1}),\label{e:thm:SLPtoprove2}
\end{align}
Due to stationary increments, it is enough to consider $t_i=t_1=t$ so that $t_{i-1}=0$.\\

Consider first $\Delta X^*_{(2)}(Tt)$. Note first that for any $\Lambda$-integrable function $f$ on $\R_+ \times \R$, it holds that (see \cite{rajput1989spectral})
\begin{equation}\label{integrationrule}
C\left\{ \zeta \ddagger \int_{\R_+ \times \R}f d\Lambda \right\} = \int_{\R_+ \times \R} \kappa_{L} (\zeta f(\xi,s)) ds d\xi.
\end{equation}
Writing the density $p$ in the form $p(x)=\alpha \widetilde{\ell}(x^{-1}) x^{\alpha-1}$, $\widetilde{\ell}(t) \sim \ell(t)$ as $t \to \infty$, we have
\begin{align}
C \left\{ \zeta \ddagger A_T^{-1} \Delta X^*_{(2)}(Tt) \right\} &= \int_0^\infty \int_{0}^{\xi T t} \kappa_L \left( \zeta A_T^{-1} \int_{s/\xi}^{Tt} e^{-\xi u + s} du \right) ds \pi(d\xi)\nonumber\\
&\hspace{-3em} = \int_0^\infty \int_{0}^{\xi T t} \kappa_L \left( \zeta A_T^{-1} \xi^{-1} \left( 1 - e^{-\xi Tt+s} \right) \right) ds \pi(d\xi)\nonumber\\
&\hspace{-3em} = \int_0^\infty \int_{0}^{t} \kappa_L \left( \zeta A_T^{-1} \xi^{-1} \left( 1 - e^{-\xi T (t-s)} \right) \right) \xi T ds \pi(d\xi)\nonumber\\
&\hspace{-3em} = \int_0^\infty \int_{0}^{t} \kappa_L \left( \zeta A_T^{-1} \xi^{-1} \left( 1 - e^{-\xi T (t-s)} \right) \right) \alpha \widetilde{\ell}(\xi^{-1}) \xi^{\alpha} T ds d\xi.\label{e:thmSPL22}
\end{align}
Suppose that $\zeta>0$, the proof is analogous in the other case. By the change of variables $x=\zeta A_T^{-1} \xi^{-1}$ in \eqref{e:thmSPL22} we have
\begin{align*}
C&\left\{ \zeta \ddagger A_T^{-1} \Delta X^*_{(2)}(Tt) \right\}\\
&= \zeta^{1+\alpha} \int_0^\infty \int_{0}^{t} \kappa_L \left( x \left( 1 - e^{-x^{-1} \frac{\zeta T}{A_T} (t-s)} \right) \right)  A_T^{-(1+\alpha)} T \widetilde{\ell}\left(A_T x \zeta^{-1}\right) \alpha x^{-\alpha-2} ds dx\\
&= \zeta^{1+\alpha} \int_0^\infty \int_{0}^{t} \kappa_L \left( x \left( 1 - e^{-x^{-1} \frac{\zeta T}{A_T} (t-s)} \right) \right) \\
&\hspace{5em} \times \frac{\widetilde{\ell}\left( T^{1/(1+\alpha)} \ell^{\#}\left(T \right)^{1/(1+\alpha)} x \zeta^{-1} \right) }{\ell^{\#} \left(T\right)} \alpha x^{-\alpha-2} ds dx.
\end{align*}
Since $T/A_T\to \infty$ as $T\to \infty$, we have that
\begin{equation*}
\kappa_L \left( x \left( 1 - e^{-x^{-1} \frac{\zeta T}{A_T} (t-s)} \right) \right) \to \kappa_L ( x ).
\end{equation*}
Due to slow variation of $\ell$, $\ell\sim \widetilde{\ell}$ and \eqref{e:thmSPL2dBcon}, we have
\begin{align*}
&\frac{\widetilde{\ell}\left( T^{1/(1+\alpha)} \ell^{\#}\left(T \right)^{1/(1+\alpha)} x \zeta^{-1} \right) }{\ell^{\#} \left(T\right)} \frac{\ell\left( T^{1/(1+\alpha)} \ell^{\#}\left(T \right)^{1/(1+\alpha)}\right) }{\ell\left( T^{1/(1+\alpha)} \ell^{\#}\left(T \right)^{1/(1+\alpha)}\right)}\\
& \hspace{5em} \sim \frac{\ell\left( \left( T\ell^{\#}\left(T \right) \right)^{1/(1+\alpha)} \right) }{\ell^{\#} \left(T\right)} \to 1,
\end{align*}
as $T \to \infty$. Hence, if the limit could be passed under the integral, we would get that
\begin{equation}\label{e:thm:SLP3}
C\left\{ \zeta \ddagger A_T^{-1} \Delta X^*_{(2)}(Tt) \right\} \to t \zeta^{1+\alpha} \int_0^\infty \kappa_L (x) \alpha x^{-\alpha-2} dx.
\end{equation}
From \eqref{kappacumfun1} with $b=0$ and the relation 
\begin{equation*}
\int_0^\infty \left(e^{ \mp iu} -1\pm iu\right)  u^{-\gamma-1} du = \exp \left\{  \mp \frac{1}{2} i \pi \gamma \right\} \frac{\Gamma(2-\gamma)}{\gamma (\gamma-1)}
\end{equation*}
valid for $1<\gamma<2$ (see e.g.~\cite[Theorem 2.2.2]{ibragimov1971independent}), we would then obtain after some computation with $\gamma=1+\alpha$,
\begin{align*}
\alpha & \int_0^\infty \kappa_L (x)  x^{-\alpha-2} dx = \alpha \int_{-\infty}^\infty \int_0^{\infty} \left(e^{ixy} -1-ixy\right)  x^{-\alpha-2} dx \mu_L(dy)\\
&= \alpha \int_{0}^\infty \int_0^{\infty} \left(e^{iu} -1-iu\right)  u^{-\alpha-2} du y^{1+\alpha}\mu_L(dy)\\
&\hspace{6em}  + \alpha \int_{-\infty}^0 \int_0^{\infty} \left(e^{-iu} -1+iu\right)  u^{-\alpha-2} du (-y)^{1+\alpha}\mu_L(dy)\\
&= \frac{\alpha\Gamma (1-\alpha)}{(1+\alpha)\alpha}  \left( e^{i(1+\alpha) \pi /2} \int_{0}^\infty y^{1+\alpha}\mu_L(dy) +  e^{-i(1+\alpha) \pi /2} \int_{-\infty}^0 |y|^{1+\alpha}\mu_L(dy) \right)\\
&= - \frac{\Gamma (1-\alpha)}{-\alpha}\\
&\hspace{1em} \times \Bigg( \cos \left(\frac{\pi (1+\alpha)}{2}\right) \left( \frac{\alpha}{1+\alpha} \int_{-\infty}^0 |y|^{1+\alpha} \mu_L(dy)  +  \frac{\alpha}{1+\alpha} \int_{0}^\infty y^{1+\alpha} \mu_L(dy)\right)\\
&\hspace{1em} -  i \sin \left(\frac{\pi (1+\alpha)}{2}\right) \left( \frac{\alpha}{1+\alpha} \int_{-\infty}^0 |y|^{1+\alpha} \mu_L(dy) - \frac{\alpha}{1+\alpha} \int_{0}^\infty y^{1+\alpha} \mu_L(dy) \right)\Bigg) \\
&= - \omega(\zeta; 1+ \alpha, c^-_\alpha, c^+_\alpha),
\end{align*}
where $\omega$ is defined in \eqref{omega} and $c^-_\alpha$, $c^+_\alpha$ in \eqref{c+-alpha}. The last equality holds because we suppose $\zeta>0$ and hence $\sign (\zeta) = 1$.

It remains to justify taking the limit under the integral in \eqref{e:thm:SLP3}. This can be done similarly as in \cite{philippe2014contemporaneous}. First, from Potter's bounds \cite[Theorem 1.5.6]{bingham1989regular}, for $0<\delta<\min \left\{ 1+\alpha - \beta_{BG}, 1-\alpha , \alpha\right\}$ there is $C_1$ such that
\begin{equation*}
\frac{\widetilde{\ell}\left( T^{1/(1+\alpha)} \ell^{\#}\left(T \right)^{1/(1+\alpha)} x \zeta^{-1} \right) }{\ell\left( T^{1/(1+\alpha)} \ell^{\#}\left(T \right)^{1/(1+\alpha)}\right)} \leq C_1 \max \left\{ x^{-\delta} \zeta^{\delta}, x^{\delta} \zeta^{-\delta} \right\}.
\end{equation*}
Hence, from \eqref{e:thmSPL2dBcon} we have that for $T$ large enough
\begin{equation}\label{e:thmSPLPotterbounds}
\frac{\widetilde{\ell}\left( T^{1/(1+\alpha)} \ell^{\#}\left(T \right)^{1/(1+\alpha)} x \zeta^{-1} \right) }{\ell^{\#} \left(T\right)} \leq C_2 \max \left\{ x^{-\delta} \zeta^{\delta}, x^{\delta} \zeta^{-\delta} \right\}.
\end{equation}
Next, note that we can bound $|\kappa_{L}(x)| \leq  \kappa_{L,1}(x) + \kappa_{L,2}(x)$ where $\kappa_{L,1}(x) = x^2 \int_{|y|\leq 1/|x|} y^2 \mu_L(dy)$ and $\kappa_{L,2}(x) = 2 |x| \int_{|y|> 1/|x|} |y| \mu_L(dy)$. Moreover,
\begin{equation}\label{kappaL1ntegrable}
\begin{aligned}
&\int_0^\infty \kappa_{L,1}(x) x^{-\alpha-2} \max \left\{ x^{-\delta} , x^{\delta} \right\} dx\\
&= \int_{|y|\leq 1} y^2 \mu_L(dy) \int_0^{1} x^{-\alpha-\delta} dx + \int_{|y|\leq 1} y^2 \mu_L(dy) \int_1^{1/|y|} x^{-\alpha+\delta} dx\\
& \hspace{4em} + \int_{|y|> 1} y^2 \mu_L(dy) \int_0^{1/|y|} x^{-\alpha-\delta} dx\\
& \leq C_3 + C_4 \int_{|y|\leq 1} |y|^{1+\alpha-\delta} \mu_L(dy)  + C_5 \int_{|y|\leq 1} |y|^{1+\alpha+\delta} \mu_L(dy)< \infty
\end{aligned}
\end{equation}
and
\begin{equation}\label{kappaL2integrable}
\begin{aligned}
&\int_0^\infty \kappa_{L,2}(x) x^{-\alpha-2} \max \left\{ x^{-\delta} , x^{\delta} \right\} dx\\
&= \int_{|y|\leq 1} |y| \mu_L(dy) \int_{1/|y|}^{\infty} x^{-\alpha-1+\delta} dx + \int_{|y|> 1} |y| \mu_L(dy) \int_1^{\infty} x^{-\alpha-1+\delta} dx\\
& \hspace{4em} + \int_{|y|> 1} |y| \mu_L(dy) \int_{1/|y|}^1 x^{-\alpha-1-\delta} dx\\
& \leq C_6 + C_7 \int_{|y|\leq 1} |y|^{1+\alpha-\delta} \mu_L(dy)  + C_8 \int_{|y|\leq 1} |y|^{1+\alpha+\delta} \mu_L(dy)< \infty
\end{aligned}
\end{equation}
by the choice of $\delta$.

Let $g_T(\zeta, x,s)=e^{-x^{-1} \frac{\zeta T}{A_T} (t-s)}$ and split $C\left\{ \zeta \ddagger A_T^{-1} \Delta X^*_{(2)}(Tt) \right\}$ into two parts:
\begin{equation}\label{e:splitI}
C\left\{ \zeta \ddagger A_T^{-1} \Delta X^*_{(2)}(Tt) \right\} = I_{T,1} + I_{T,2},
\end{equation}
where
\begin{align}
&\begin{split}\label{e:splitI1}
I_{T,1} &= \zeta^{1+\alpha} \int_0^\infty \int_{0}^{t} \kappa_L \left( x \left( 1 - g_T(\zeta, x,s) \right) \right) \frac{\widetilde{\ell}\left( T^{1/(1+\alpha)} \ell^{\#}\left(T \right)^{1/(1+\alpha)} x \zeta^{-1} \right) }{\ell^{\#} \left(T\right)}\\
&\qquad \qquad \qquad \times \alpha x^{-\alpha-2} \1_{[0,1/2]}(g_T(\zeta, x,s)) ds dx,\\
\end{split}\\
&\begin{split}\label{e:splitI2}
I_{T,2} &= \zeta^{1+\alpha} \int_0^\infty \int_{0}^{t} \kappa_L \left( x \left( 1 - g_T(\zeta, x,s) \right) \right)  \frac{\widetilde{\ell}\left( T^{1/(1+\alpha)} \ell^{\#}\left(T \right)^{1/(1+\alpha)} x \zeta^{-1} \right) }{\ell^{\#} \left(T\right)} \\
&\qquad \qquad \qquad \times \alpha x^{-\alpha-2} \1_{[1/2,1]}(g_T(\zeta, x,s)) ds dx.
\end{split}
\end{align}
We have that 
\begin{equation}\label{kappaL1bar}
\sup_{1/2 \leq c \leq 1} \kappa_{L,1}(c x) \leq x^2 \int_{|y|\leq 2/|x|} y^2 \mu_L(dy) =: \overline{\kappa}_{L,1}(x),
\end{equation}
where
\begin{equation}\label{kappaL1barintegrable}
\int_0^\infty \overline{\kappa}_{L,1}(x) x^{-\alpha-2} \max \left\{ x^{-\delta} , x^{\delta} \right\} dx < \infty
\end{equation}
by the same argument as in \eqref{kappaL1ntegrable}. Furthermore, we have that 
\begin{equation*}
\sup_{1/2 \leq c \leq 1} \kappa_{L,2}(c x) \leq \kappa_{L,2}(x),
\end{equation*}
and hence $|\kappa_L \left( x \left( 1 - g_T(\zeta,x,s) \right) \right) \1_{[0,1/2]}(g_T(\zeta,x,s)) | \leq \overline{\kappa}_{L,1}(x) + \kappa_{L,2}(x)$. By combining with \eqref{e:thmSPLPotterbounds}, we end up with the upper bound which is integrable by \eqref{kappaL2integrable} and \eqref{kappaL1barintegrable}. Hence the dominated convergence theorem may be applied to $I_{T,1}$ showing that $I_{T,1}$ converges to the limit in \eqref{e:thm:SLP3}.

We next show that $I_{T,2} \to 0$. Using the inequality 
\begin{equation*}
\left|e^{ix}-\sum_{k=0}^{n} \frac{(ix)^k}{k!} \right| \leq \min \left\{ \frac{|x|^{n+1}}{(n+1)!}, \frac{2 |x|^n}{n!} \right\},
\end{equation*}
we get by \eqref{kappacumfun1} that for any $x \in \R$,
\begin{equation*}
|\kappa_{L}(x)| \leq \int_{\R} \left| e^{ixy}-1-ixy\right| \mu_L(dy) \leq \int_{|xy|\leq 1} |x y|^2 \mu_L(dy) + 2 \int_{|xy|> 1} |x y| \mu_L(dy).
\end{equation*}
Then,  by taking $\gamma$ such that 
\begin{equation}\label{e:gam}
\max \{ \beta_{BG}, 1\}<\gamma<1+\alpha,
\end{equation}
we get
\begin{equation}\label{e:thmSLPboundonkappa}
|\kappa_{L}(x)| \leq \int_{|xy|\leq 1} |x y|^\gamma \mu_L(dy) + 2 \int_{|xy|> 1} |x y|^{\gamma} \mu_L(dy) \leq C_1 |x|^\gamma,
\end{equation}
since $\int_{\R}|y|^{\gamma} \mu_L(dy) < \infty$. Now since $\1_{[1/2,1]}(g_T(\zeta, x,s)) = \1_{\left[ \frac{\zeta  (t-s) T}{A_T \log 2 }, \infty \right)}(x)$, we have by using \eqref{e:thmSPLPotterbounds} for $\delta<1+\alpha-\gamma$
\begin{equation}\label{e:thmSLPproofIT2zero}
\begin{aligned}
\left| I_{T,2} \right| &\leq C_2 \int_0^\infty \int_{0}^{t} x^{\gamma-\alpha-2}  \max \left\{ x^{-\delta} , x^{\delta} \right\} \1_{\left[ \frac{\zeta  (t-s) T}{A_T \log 2}, \infty \right)}(x) ds dx\\
&= C_2 \int_{0}^{t} \int_0^1 x^{\gamma-\alpha-2-\delta} \1_{\left[ \frac{\zeta u T}{A_T \log 2}, \infty \right)}(x)  dx du\\
&\hspace{5em} + C_2 \int_{0}^{t} \int_1^\infty x^{\gamma-\alpha-2+\delta} \1_{\left[ \frac{\zeta u T}{A_T \log 2}, \infty \right)}(x)  dx du\\
&= C_3 \int_{0}^{t} \1_{\left(0, \frac{A_T \log 2}{\zeta T} \right]}(u) du - C_4 \left(\frac{T}{A_T}\right)^{\gamma-\alpha-1-\delta} \int_0^t u^{\gamma-\alpha-1-\delta} \1_{\left(0, \frac{A_T \log 2}{\zeta T} \right]}(u) du\\
&\qquad \qquad + C_6 \left(\frac{T}{A_T}\right)^{\gamma-\alpha-1+\delta} \int_0^t u^{\gamma-\alpha-1+\delta} \1_{\left[\frac{A_T  \log 2}{\zeta T}, \infty \right)}(u) du \to 0,
\end{aligned}
\end{equation}
as $T\to \infty$, which completes the proof of \eqref{e:thm:SLPtoprove2}.\\

To complete the proof, it remains to show \eqref{e:thm:SLPtoprove1}. From \eqref{integrationrule} and by making change of variables we get that 
\begin{align*}
C\left\{ \zeta \ddagger A_T^{-1} \Delta X^*_{(1)}(Tt) \right\} &= \int_0^\infty \int_{-\infty}^{0} \kappa_L \left( \zeta A_T^{-1} \int_{0}^{Tt} e^{-\xi u + s} du \right) ds \pi(d\xi)\\
&\hspace{-6em}= \int_0^\infty \int_{-\infty}^{0} \kappa_L \left( \zeta A_T^{-1} e^s \xi^{-1} \left(1-e^{-\xi Tt} \right)\right) ds \pi(d\xi)\\
&\hspace{-6em}= \int_0^\infty \int_{-\infty}^{0} \kappa_L \left( \zeta T A_T^{-1} e^s \xi^{-1} \left(1-e^{-\xi t} \right)\right) ds \pi(T^{-1} d\xi)\\
&\hspace{-6em}= \int_0^\infty \int_{-\infty}^{0} \kappa_L \left( \zeta T A_T^{-1} e^s \xi^{-1} \left(1-e^{-\xi t} \right)\right) \alpha \widetilde{\ell}(T \xi^{-1}) \xi^{\alpha-1} T^{-\alpha} ds d\xi.
\end{align*}
By using Potter's bounds \cite[Theorem 1.5.6]{bingham1989regular} we have for $\delta>0$
\begin{equation*}
\widetilde{\ell}(T \xi^{-1}) = \frac{\widetilde{\ell}(T \xi^{-1})}{\widetilde{\ell}( \xi^{-1})} \widetilde{\ell}(\xi^{-1}) \leq C \max \left\{ T^{-\delta}, T^{\delta} \right\} \widetilde{\ell}(\xi^{-1}).
\end{equation*}
Taking $\gamma$ as in \eqref{e:gam} and using the bound in \eqref{e:thmSLPboundonkappa}, we get that
\begin{align}
&\left| C\left\{ \zeta \ddagger A_T^{-1} \Delta X^*_{(1)}(Tt) \right\} \right|\nonumber\\
&\leq C_2 |\zeta|^\gamma T^{\gamma-\alpha+\delta} A_T^{-\gamma} \int_0^\infty \int_{-\infty}^{0} e^{\gamma s} \left(\xi^{-1} \left(1-e^{-\xi t} \right) \right)^\gamma \widetilde{\ell}(\xi^{-1}) \xi^{\alpha-1} ds d\xi \nonumber\\
&\leq C_2 |\zeta|^\gamma T^{\gamma-\alpha+\delta} A_T^{-\gamma} \int_0^\infty \gamma^{-1} \widetilde{\ell}(\xi^{-1}) \xi^{\alpha-1} d\xi \to 0 \text{ as } T \to \infty,\label{e:thmSLPproofX1zero}
\end{align}
if we take $\delta$ small enough so that $\gamma-\alpha+\delta-\gamma/(1+\alpha)<0$.
\end{proof}

\bigskip

\begin{proof}[Proof of Theorem \ref{thm:Zcase}]
The proof relies on the following two facts proved in \cite[Eqs.~(41)-(42)]{philippe2014contemporaneous} (see also \cite[Theorem 4.15]{bismut1983calcul}) which follow from \eqref{LevyMCond}:
\begin{align}
\lim_{\lambda \to 0} \lambda \kappa_L \left( \lambda^{-1/\beta} \zeta \right) &= - |\zeta|^{\beta} \omega (\zeta; \beta, c^+, c^-), \qquad \text{for any } \zeta \in \R \label{PPSprop1}\\
|\kappa_L(\zeta)| &\leq C |\zeta|^{\beta},  \qquad \text{for any } \zeta \in \R.\label{PPSprop2}
\end{align}
Here, $c^+$, $c^-$ and $\beta$ are constants from \eqref{LevyMCond}. Note that from \eqref{mff} we can write for $\xi>0$
\begin{equation*}
\mff (\xi, t-s) - \mff(\xi, -s) = \begin{cases}
e^{\xi s} - e^{-\xi (t-s)}, & \text{ if } s<0,\\
1- e^{-\xi (t-s)}, & \text{ if } 0\leq s < t,\\
0, & \text{ if } s\geq t.\\
\end{cases}
\end{equation*}
Using this and the change of variables in \eqref{e:cumfidisnormLform1}, we get for $\zeta_1,\dots,\zeta_m \in \R$ and $t_1<\cdots<t_m$ that
\begin{align*}
&C \left\{ \zeta_1, \dots, \zeta_m \ddagger \left(T^{-1+\alpha/\beta} \ell(T)^{-1/\beta} X^*( T t_1),\dots , T^{-1+\alpha/\beta} \ell(T)^{-1/\beta} X^*( T t _m) \right) \right\} \\
&= \int_0^\infty \int_{-\infty}^\infty \kappa_L \left( T^{-1+\alpha/\beta} \ell(T)^{-1/\beta} \sum_{j=1}^m \zeta_j \xi^{-1} \left( \mathfrak{f}(\xi, T t_j-s) - \mathfrak{f}(\xi, -s)  \right) \right) ds \xi \pi(d\xi)\\
&= \int_0^\infty \int_{-\infty}^\infty \kappa_L \left( T^{\alpha/\beta} \ell(T)^{-1/\beta} \sum_{j=1}^m \zeta_j \xi^{-1}  \left( \mathfrak{f}(T^{-1} \xi, T t_j-s) - \mathfrak{f}(T^{-1} \xi, -s)  \right) \right)\\
&\hspace{12em} \times ds T^{-1} \xi \pi(T^{-1} d\xi)\\
&= \int_0^\infty \int_{-\infty}^\infty \kappa_L \left( T^{\alpha/\beta} \ell(T)^{-1/\beta} \sum_{j=1}^m \zeta_j \xi^{-1} \left( \mathfrak{f}(\xi, t_j-T^{-1}s) - \mathfrak{f}( \xi, -T^{-1} s)  \right) \right)\\
&\hspace{12em} \times ds T^{-1} \xi \pi(T^{-1} d\xi)\\
&= \int_0^\infty \int_{-\infty}^\infty \kappa_L \left( \left(T^{-\alpha} \ell(T)\right)^{-1/\beta} \sum_{j=1}^m \zeta_j \xi^{-1} \left( \mathfrak{f}(\xi, t_j-s) - \mathfrak{f}( \xi, -s)  \right) \right) ds \xi \pi(T^{-1} d\xi).
\end{align*}
Again, because of \eqref{regvarofp}, we can write $p$ in the form $p(x)=\alpha \widetilde{\ell}(x^{-1}) x^{\alpha-1}$ with $\widetilde{\ell}$ slowly varying at infinity such that $\widetilde{\ell}(t) \sim \ell(t)$ as $t \to \infty$. Now we have
\begin{align}
C & \left\{ \zeta_1, \dots, \zeta_m \ddagger \left(T^{-1+\alpha/\beta} \ell(T)^{-1/\beta} X^*( T t_1),\dots , T^{-1+\alpha/\beta} \ell(T)^{-1/\beta} X^*( T t _m) \right) \right\} \nonumber\\
&= \int_0^\infty \int_{-\infty}^\infty \kappa_L \left( \left(T^{-\alpha} \ell(T)\right)^{-1/\beta}  \xi^{-1} \sum_{j=1}^m \zeta_j  \left( \mathfrak{f}(\xi, t_j-s) - \mathfrak{f}( \xi, -s)  \right) \right) \nonumber\\
&\hspace{12em} \times \alpha T^{-\alpha} \widetilde{\ell}(T \xi^{-1}) \xi^{\alpha}  ds d\xi \label{e:proof:intformoment1} \\
\begin{split}
&= - \int_0^\infty \int_{-\infty}^\infty \left| \sum_{j=1}^m \zeta_j  \left( \mathfrak{f}(\xi, t_j-s) - \mathfrak{f}( \xi, -s)  \right) \right|^\beta \label{e:proof:intH}\\
&\qquad \times \omega \left( \sum_{j=1}^m \zeta_j  \left( \mathfrak{f}(\xi, t_j-s) - \mathfrak{f}( \xi, -s)  \right); \beta, c^+, c^- \right) h_T(\xi, s) \alpha \xi^{\alpha-\beta} ds d\xi,
\end{split}
\end{align}
where
\begin{align*}
&h_T(\xi,s) \\
&= - \frac{ \kappa_L \left( \left(T^{-\alpha} \ell(T) \xi^{\beta}\right)^{-1/\beta} \sum_{j=1}^m \zeta_j  \left( \mathfrak{f}(\xi, t_j-s) - \mathfrak{f}( \xi, -s)  \right) \right) T^{-\alpha} \ell(T) \xi^{\beta} \frac{\widetilde{\ell}(T \xi^{-1}) }{\ell(T)}}{\left| \sum_{j=1}^m \zeta_j  \left( \mathfrak{f}(\xi, t_j-s) - \mathfrak{f}( \xi, -s)  \right) \right|^\beta \omega \left( \sum_{j=1}^m \zeta_j  \left( \mathfrak{f}(\xi, t_j-s) - \mathfrak{f}( \xi, -s)  \right); \beta, c^+, c^- \right)}.
\end{align*}
By taking $\lambda=\lambda(T,\xi)=T^{-\alpha} \ell(T) \xi^\beta$ in \eqref{PPSprop1} and using slow variation of $\ell$, we conclude that $h_T(\xi,s) \to 1$ as $T\to \infty$ for each $\xi>0$, $s\in \R$. 

It remains to show that the dominated convergence theorem can be applied to get \eqref{e:Zalphabetafidis}. By using \eqref{PPSprop2} we have that
\begin{align*}
&\left| h_T(\xi,s) \right|\\
&\leq \frac{ C \left(T^{-\alpha} \ell(T) \xi^{\beta}\right)^{-1} \left|\sum_{j=1}^m \zeta_j  \left( \mathfrak{f}(\xi, t_j-s) - \mathfrak{f}( \xi, -s)  \right) \right|^\beta T^{-\alpha} \ell(T) \xi^{\beta} \frac{\widetilde{\ell}(T \xi^{-1}) }{\ell(T)}}{\left| \sum_{j=1}^m \zeta_j  \left( \mathfrak{f}(\xi, t_j-s) - \mathfrak{f}( \xi, -s)  \right) \right|^\beta \left|\omega \left( \sum_{j=1}^m \zeta_j  \left( \mathfrak{f}(\xi, t_j-s) - \mathfrak{f}( \xi, -s)  \right); \beta, c^+, c^- \right)\right|}.
\end{align*}
Since $|\omega(z;\beta, c^+, c^-)|$ does not depend on $z$, we have that
\begin{equation*}
\left| h_T(\xi,s) \right| \leq C_1 \frac{\widetilde{\ell}(T \xi^{-1}) }{\ell(T)}.
\end{equation*}
By Potter's bounds \cite[Theorem 1.5.6]{bingham1989regular}, for any $\delta>0$ there is $C_2$ such that $\widetilde{\ell}(T \xi^{-1})/\ell(T) \leq C_1 \max \left\{ \xi^{-\delta}, \xi^{\delta}\right\}$. Taking $\delta$ small enough, we get a bound for the function under the integral in \eqref{e:proof:intH} which is integrable.
\end{proof}

\bigskip

\begin{proof}[Proof of Theorem \ref{thm:BMcase}]
Let $0=t_0<t_1<\cdots<t_m$, $\zeta_1,\dots,\zeta_m \in \R$. Similarly as in the proof of Theorem \ref{thm:SLPcase}, to prove that
\begin{equation*}
\sum_{i=1}^m \zeta_i T^{-1/2} \left( X^*(Tt_{i}) - X^*(Tt_{i-1}) \right) \overset{d}{\to} \sum_{i=1}^m \zeta_i \widetilde{\sigma} \left( B (t_i) - B (t_{i-1}) \right).
\end{equation*}
it will be enough to prove that
\begin{align}
T^{-1/2} \Delta X^*_{(1)}(Tt_{i}) &\overset{d}{\to} 0,\label{e:thm:BMtoprove1}\\
T^{-1/2} \Delta X^*_{(2)}(Tt_{i}) &\overset{d}{\to}  \widetilde{\sigma} \left( B(t_i) - B(t_{i-1} \right).\label{e:thm:BMtoprove2}
\end{align}
where $\Delta X^*_{(1)}$ and $\Delta X^*_{(2)}$ are defined in \eqref{e:thmSLP:decomposition}. Due to stationary increments, it is enough to consider $t_i=t_1=t$ so that $t_{i-1}=0$.\\

A change of variables and \eqref{integrationrule} give 
\begin{align*}
C\left\{ \zeta \ddagger T^{-1/2} \Delta X^*_{(1)}(Tt) \right\} &= \int_0^\infty \int_{-\infty}^{0} \kappa_L \left( \zeta T^{-1/2} \int_{0}^{Tt} e^{-\xi u + s} du \right) ds \pi(d\xi)\\
&= \int_0^\infty \int_{-\infty}^{0} \kappa_L \left( \zeta T^{-1/2} e^s \xi^{-1} \left(1-e^{-\xi Tt} \right)\right) ds \pi(d\xi).
\end{align*}
Since for any $\xi>0$, $s<0$, $ \kappa_L \left( \zeta T^{-1/2} e^s \xi^{-1} \left(1-e^{-\xi Tt} \right)\right) \to 0$ as $T\to \infty$, it remains to show that the dominated convergence theorem is applicable. By \eqref{e:bd}, we get for any $\zeta \in \R$, $|\kappa_{L} (\zeta) |\leq \frac{1}{2}  \zeta^2 \int_{\R} x^2 \mu_L(dx) = C \zeta^2$. Hence, we have
\begin{equation*}
\left| \kappa_L \left( \zeta T^{-1/2} e^s \xi^{-1} \left(1-e^{-\xi Tt} \right)\right) \right| \leq C \zeta^2 T^{-1} e^{2s} \xi^{-2} \left(1-e^{-\xi Tt}\right)^2
\end{equation*}
and 
\begin{align*}
&\int_0^\infty \int_{-\infty}^{0} \zeta^2 T^{-1} e^{2s} \xi^{-2} \left(1-e^{-\xi Tt}\right)^2 ds \pi(d\xi)\\
&\hspace{5em}= \zeta^2 \int_0^\infty t^2 T \left(\xi T t\right)^{-2} \left(1-e^{-\xi Tt}\right)^2 ds \pi(d\xi)\\
&\hspace{5em}\leq \zeta^2 t \int_0^\infty \xi^{-1} \pi(d\xi) < \infty,
\end{align*}
since $(1-e^{-x})^2/x^2 \leq x^{-1}$, $x >0$. This completes the proof of \eqref{e:thm:BMtoprove1}.\\

Next, for $\Delta X^*_{(2)}(Tt)$ we have from \eqref{integrationrule}
\begin{align}
C\left\{ \zeta \ddagger T^{-1/2} \Delta X^*_{(2)}(Tt) \right\} &= \int_0^\infty \int_{0}^{\xi T t} \kappa_L \left( \zeta T^{-1/2} \int_{s/\xi}^{Tt} e^{-\xi u + s} du \right) ds \pi(d\xi)\nonumber\\
&\hspace{-5em}= \int_0^\infty \int_{0}^{\xi T t} \kappa_L \left( \zeta T^{-1/2} \xi^{-1} \left( 1 - e^{-\xi Tt+s} \right) \right) ds \pi(d\xi)\nonumber\\
&\hspace{-5em}= \int_0^\infty \int_{0}^{t} \kappa_L \left( \zeta T^{-1/2} \xi^{-1} \left( 1 - e^{-\xi T (t-s)} \right) \right) \xi T ds \pi(d\xi)\nonumber\\
&\hspace{-5em}=-\sigma^2 \zeta^2 \int_0^\infty \int_{0}^{t} h_T(\xi, s, \zeta) \xi^{-1} \left( 1 - e^{-\xi T (t-s)} \right)^2 ds \pi(d\xi),\label{e:thmBM111}
\end{align}
where
\begin{equation*}
h_T(\xi, s, \zeta) = - \frac{\kappa_L \left( \zeta T^{-1/2} \xi^{-1} \left( 1 - e^{-\xi T (t-s)} \right) \right)}{\sigma^2 \zeta^2 T^{-1} \xi^{-2} \left( 1 - e^{-\xi T (t-s)} \right)^2 }.
\end{equation*}
From \eqref{e:bd} we get that
\begin{equation*}
\left|h_T(\xi, s, \zeta) \xi^{-1} \left( 1 - e^{-\xi T (t-s)} \right)^2\right| = \frac{\left|\kappa_L \left( \zeta T^{-1/2} \xi^{-1} \left( 1 - e^{-\xi T (t-s)} \right) \right)\right| }{\sigma^2 \zeta^2 T^{-1} \xi^{-1} } \leq \frac{C}{\sigma^2} \xi^{-1},
\end{equation*}
and hence, the dominated convergence theorem can be applied. By using \eqref{kappaLtoX}, we have that $\Var L(1) = \kappa_{L}''(0) = 2 \kappa_{X}''(0) = 2 \sigma^2$. Since $\Var L(1)< \infty$ and $\E L(1)=0$, we can expand $\kappa_L(\zeta)=-\sigma^2 \zeta^2+o(|\zeta|^2)$ as $\zeta\to 0$. Now it follows that $-\kappa_L(\zeta)/(\sigma^2 \zeta^2) \to 1$ as $\zeta\to 0$ and hence
\begin{equation*}
h_T(\xi, s, \zeta) \xi^{-1} \left( 1 - e^{-\xi T (t-s)} \right)^2 \to \xi^{-1}
\end{equation*}
as $T\to \infty$. From \eqref{e:thmBM111} we conclude that
\begin{equation*}
C\left\{ \zeta \ddagger T^{-1/2} \Delta X^*_{(2)}(Tt) \right\} \to -\sigma^2 \zeta^2 t \int_0^\infty \xi^{-1} \pi(d\xi).
\end{equation*}
\end{proof}


\bigskip

\section{Proofs of weak convergence in function space}\label{s:proof-weak}

A useful formula which will be used many times is given in the following lemma (see \cite[Lemma 2]{von1965inequalities}).

\begin{lemma}\label{lemma:bound}
Let $Y$ be a random variable with characteristic function $\phi(\zeta)$ and moment $\E|Y|^r<\infty$, $0<r<2$. Then 
\begin{equation}\label{e:b1}
\E |Y|^r = k_r \int_{-\infty}^{\infty} \left( 1- \Re \phi(\zeta) \right) |\zeta|^{-r-1} d \zeta,
\end{equation}
where 
\begin{equation*}
k_r = 	\frac{\Gamma(r+1)}{\pi} \sin \left(\frac{r\pi}{2}\right)>0.
\end{equation*}
\end{lemma}

In particular, if $Y$ is symmetric $\beta$-stable, $0<\beta<2$, with characteristic function $\phi(\zeta)=\exp\{s^\beta |\zeta|^\beta\}$, $s>0$, then for $0<r<\beta$
\begin{equation}\label{e:b2}
\E |Y|^r = k_r \int_{-\infty}^{\infty} \left( 1- \exp \{s^\beta |\zeta|^\beta\} \right) |\zeta|^{-r-1} d \zeta.
\end{equation}

\begin{proof}[Proof of Theorem \ref{thm:funconv}]
For an integrated process $X^*$ and normalizing sequence $A_T$, by \cite[Theorem 12.3, Eq.~(12.51)]{billingsley1968} and stationarity of increments it is enough to prove that for some $C>0$, $T_0\geq 1$, $\gamma >0$ and $a>1$, the bound
\begin{equation}\label{tightnessmomcrit}
\E \left|A_T^{-1} X^*(Tt)\right|^\gamma \leq C t^a,
\end{equation}
holds for all $t\in[0,1]$ and $T\geq T_0$.\\

\textit{For the case of Theorem \ref{thm:FBMcase}}, we take the second derivative of $\kappa_{X^*(Tt)}(A_T^{-1}\zeta)$ given by \eqref{e:cumfidisnorm1} with respect to $\zeta$ and let $\zeta\to 0$ to get that the variance is
\begin{equation}\label{varianceformula}
\E \left(A_T^{-1} X^*(Tt)\right)^2 = 2 \E \left(X(1)\right)^2 A_T^{-2} \int_0^\infty \int_0^{T t}\left( 1 -e^{-\xi (T t - s)} \right) ds  \xi^{-1} \pi(d\xi).
\end{equation}
By \eqref{regvarofp}, we can write $p$ in the form $p(x)=\alpha \widetilde{\ell}(x^{-1}) x^{\alpha-1}$ with $\widetilde{\ell}$ slowly varying at infinity such that $\widetilde{\ell}(t) \sim \ell(t)$ as $t \to \infty$. Since $A_T=T^{1-\alpha/2 } \ell(T)^{1/2}$, we have by the change of variables
\begin{align*}
\E \left(A_T^{-1} X^*(Tt)\right)^2 &= 2 \E \left(X(1)\right)^2 t T^{-1+\alpha} \int_0^\infty \int_0^{1}\left( 1 -e^{-\xi Tt u} \right) du  \alpha \frac{\widetilde{\ell}(\xi^{-1})}{\ell(T)} \xi^{\alpha-2} d\xi\\
 &= 2 \alpha \E \left(X(1)\right)^2 t^{2-\alpha} \int_0^\infty \int_0^{1}\left( 1 -e^{-x u} \right) \frac{\widetilde{\ell}(Tt x^{-1})}{\ell(T)} x^{\alpha-2} du dx.
\end{align*}
By using Potter's bounds \cite[Theorem 1.5.6]{bingham1989regular}, for arbitrary $\delta>0$ we obtain
\begin{align}
\E \left(A_T^{-1} X^*(Tt)\right)^2 &\leq C_1 t^{2-\alpha} \int_0^\infty \int_0^{1}\left( 1 -e^{-x u} \right) \max \left\{ \left(t/x\right)^\delta, \left(t/x\right)^{-\delta} \right\} x^{\alpha-2} du dx \nonumber\\
&\leq C_1 t^{2-\alpha-\delta} \int_0^\infty \int_0^{1}\left( 1 -e^{-x u} \right) \max \left\{ x^\delta, x^{-\delta} \right\} x^{\alpha-2} du dx\label{e:finiteint}\\
&\leq C_2 t^{2-\alpha-\delta},\nonumber
\end{align}
since the integral in \eqref{e:finiteint} is finite if we take $\delta$ small enough.
Hence, the tightness criterion \eqref{tightnessmomcrit} holds if we take $\delta<1-\alpha$.\\

\textit{For the case of Theorem \ref{thm:Zcase}}, note that from \eqref{e:proof:intformoment1}, by using \eqref{PPSprop2} and Potter's bounds as in the proof of Theorem \ref{thm:FBMcase}, we get
\begin{align*}
&\left| \kappa_{X^*} (A_T^{-1} \zeta, Tt) \right| \nonumber\\
&\leq \int_0^\infty \int_{-\infty}^\infty \left| \kappa_L \left( \left(T^{-\alpha} \ell(T)\right)^{-1/\beta} \zeta \xi^{-1} \left( \mathfrak{f}(\xi, t-s) - \mathfrak{f}( \xi, -s)  \right) \right) \right| \alpha T^{-\alpha} \widetilde{\ell}(T \xi^{-1}) \xi^{\alpha} ds d\xi \\
&\leq C\int_0^\infty \int_{-\infty}^\infty T^{\alpha} \ell(T)^{-1} \left|\zeta \right|^{\beta} \xi^{-\beta} \left( \mathfrak{f}(\xi, t-s) - \mathfrak{f}( \xi, -s)  \right)^{\beta} \alpha T^{-\alpha} \widetilde{\ell}(T \xi^{-1}) \xi^{\alpha} ds d\xi \\
&\leq C_1 \left|\zeta \right|^{\beta}  \int_0^\infty \int_{-\infty}^\infty \left( \mathfrak{f}(\xi, t-s) - \mathfrak{f}( \xi, -s)  \right)^{\beta} \max \left\{ \xi^{-\delta}, \xi^{\delta}\right\} \xi^{\alpha-\beta} ds d\xi,
\end{align*}
with $\mathfrak{f}$ given by \eqref{mff}. Now by the change of variables
\begin{align}
&\left| \kappa_{X^*} (A_T^{-1} \zeta, Tt) \right|\nonumber\\
&\leq C_1 t^{\beta-\alpha-1} \left|\zeta \right|^{\beta}  \int_0^\infty \int_{-\infty}^\infty \left( \mathfrak{f}(x/t, t-s) - \mathfrak{f}( x/t, -s)  \right)^{\beta} \max \left\{ (x/t)^{-\delta}, (x/t)^{\delta}\right\} x^{\alpha-\beta} ds dx\nonumber\\
&\leq C_1 t^{\beta-\alpha-\delta} \left|\zeta \right|^{\beta}  \int_0^\infty \int_{-\infty}^\infty \left( \mathfrak{f}(x/t, t-tu) - \mathfrak{f}( x/t, -tu)  \right)^{\beta} \max \left\{ x^{-\delta}, x^{\delta}\right\} x^{\alpha-\beta} du dx\nonumber\\
&= C_1 t^{\beta-\alpha-\delta} \left|\zeta \right|^{\beta}  \int_0^\infty \int_{-\infty}^\infty \left( \mathfrak{f}(x, 1-u) - \mathfrak{f}( x, -u)  \right)^{\beta} \max \left\{ x^{-\delta}, x^{\delta}\right\} x^{\alpha-\beta} du dx\nonumber\\
&= C_2 t^{\beta-\alpha-\delta} \left|\zeta \right|^{\beta}.\label{e:Zcaseweak-bound}
\end{align}
Let $\widetilde{Y}$ denote the symmetrization of random variable $Y$, i.e.~$\widetilde{Y}=Y-Y'$ with $Y'=^d Y$ and independent of $Y$. By \cite[Proposition 3.6.5]{gut2013probability}, if $\E Y=0$ and $\E|Y|^r<\infty$ for some $r \geq 1$, then it holds that $\E|Y|^r \leq \E |\widetilde{Y}|^r$. Now, since the characteristic function of the symmetrized random variable $ \widetilde{X^*}(Tt)$ is $|\exp \kappa_{X^*} (\zeta, Tt)|^2$, we have by applying Lemma \ref{lemma:bound} to $A_T^{-1} \widetilde{X^*}(Tt)$
\begin{equation}\label{momSLPcase111}
\begin{aligned}
\E \left| A_T^{-1} X^*(Tt)\right|^q &\leq \E \left| A_T^{-1} \widetilde{X^*}(Tt)\right|^q\\
&= k_q \int_{-\infty}^{\infty} \left( 1- |\exp \kappa_{X^*} (A_T^{-1} \zeta, Tt)|^2 \right) |\zeta|^{-q-1} d \zeta.
\end{aligned}
\end{equation}
Furthermore, using the inequality $\Re z\geq -|z|$, $z\in \mathbb{C}$, we have
\begin{equation*}
|\exp \kappa_{X^*} (A_T^{-1} \zeta, Tt)|^2 = \exp \{ 2 \Re \kappa_{X^*} (A_T^{-1} \zeta, Tt) \} \geq  \exp \{ - 2 |\kappa_{X^*} (A_T^{-1} \zeta, Tt) | \} ,
\end{equation*}
so that we get from \eqref{momSLPcase111} that
\begin{equation}\label{momtocumfun-inequality}
\E \left| A_T^{-1} X^*(Tt)\right|^q \leq k_q \int_{-\infty}^{\infty} \left( 1- \exp  \{ - 2 |\kappa_{X^*} (A_T^{-1} \zeta, Tt) | \} \right) |\zeta|^{-q-1} d \zeta.
\end{equation}
From the bound \eqref{e:Zcaseweak-bound} we get that for $1\leq q<\beta$
\begin{equation}\label{e:Zcasemomentbound}
\E \left| A_T^{-1} X^*(Tt)\right|^q \leq k_q \int_{-\infty}^{\infty} \left( 1- \exp  \{ - 2 C_2 t^{\beta-\alpha-\delta} \left|\zeta \right|^{\beta} \} \right) |\zeta|^{-q-1} d \zeta,
\end{equation}
By \eqref{e:b2}, the right hand side of \eqref{e:Zcasemomentbound} is the $q$-th absolute moment of a symmetric $\beta$-stable random variable with scale parameter $s=(2C_2)^{1/\beta} t^{(\beta-\alpha-\delta)/\beta}$. By using \cite[Property 1.2.17]{samorodnitsky1994stable} we obtain
\begin{equation}\label{e:6.8b}
\E \left| A_T^{-1} X^*(Tt)\right|^q \leq C_3 t^{(\beta-\alpha-\delta) q /\beta}.
\end{equation}
Taking $q>\beta/(\beta-\alpha-\delta)$ yields \eqref{tightnessmomcrit}.\\

\textit{Consider finally Theorem \ref{thm:BMcase}.} From the variance formula \eqref{varianceformula} by using $\int_0^{\infty} \xi^{-1} \pi(d\xi)<\infty$ and $A_T=T^{1/2}$ we have that
\begin{equation*}
\E \left(A_T^{-1} X^*(Tt)\right)^2 \leq  C_1 T^{-1} \int_0^\infty \int_0^{T t} ds \xi^{-1} \pi(d\xi) = C_2 t.
\end{equation*}
Similarly, by taking fourth derivative of $\kappa_{X^*(Tt)}(A_T^{-1}\zeta)$ with respect to $\zeta$ and letting $\zeta\to 0$ we get that the fourth cumulant $\kappa_{A_T^{-1} X^*(Tt)}^{(4)}$ of $\kappa_{X^*(Tt)}(A_T^{-1}\zeta)$ is
\begin{equation*}
\kappa_{A_T^{-1} X^*(Tt)}^{(4)} = 4 \kappa_{X}^{(4)} T^{-2} \int_0^\infty \int_0^{T t}\left( 1 -e^{-\xi (T t - s)} \right)^3 ds  \xi^{-3} \pi(d\xi),
\end{equation*}
where $\kappa_{X}^{(4)}$ is the fourth cumulant of $X(1)$. Now by using the assumption $\int_0^{\infty} \xi^{-2} \pi(d\xi)<\infty$, we get the bound
\begin{align*}
\kappa_{A_T^{-1} X^*(Tt)}^{(4)} &\leq 4 \kappa_{X}^{(4)} T^{-2} \int_0^\infty \int_0^{T t}\left( \frac{1 -e^{-\xi u}}{\xi u}\right) u du  \xi^{-2} \pi(d\xi)\\
&\leq  4 \kappa_{X}^{(4)} T^{-2}  \int_0^{T t} u du  \int_0^\infty \xi^{-2} \pi(d\xi)\\
&= C_3 T^{-2} \frac{(T t)^2}{2} = C_4 t^2.
\end{align*}
Finally then
\begin{equation*}
\E \left(A_T^{-1} X^*(Tt)\right)^4 = \kappa_{A_T^{-1} X^*(Tt)}^{(4)}  + 3 \left(\kappa_{A_T^{-1} X^*(Tt)}^{(2)}\right)^2 \leq C_4 t^2 + 3 C_2^2 t^2 \leq C_5 t^2,
\end{equation*}
and \eqref{tightnessmomcrit} holds.
\end{proof}

\bigskip

\section{Proofs related to intermittency}\label{s:proof-intermittency}

\begin{proof}[Proof of Theorem \ref{thm:momFBMcase}]
That $\tau_{X^*}(q)=q-\alpha$ for $q\geq 2$ follows from \eqref{tausupOUqqstar}  \cite[Theorem 7]{GLST2017Arxiv}. We will show that $\tau_{X^*}(1)=1-\frac{\alpha}{2}$. Since $\tau_{X^*}(0)=0$, $\tau_{X^*}$ is convex function (\cite[Proposition 2.1]{GLST2016JSP}) passing through three collinear points: $(0,0)$, $(1,1-\frac{\alpha}{2})$, $\left(2, 2\left(1-\frac{\alpha}{2}\right) \right)$. By \cite[Lemma 3]{GLST2017Arxiv}, $\tau_{X^*}$ is linear and $\tau_{X^*}(q)=\left( 1 - \frac{\alpha}{2}\right) q$ for $q\leq 2$ which would complete the proof.

To prove that $\tau_{X^*}(1)=1-\frac{\alpha}{2}$, let $X^*_1(t)$ and $X^*_2(t)$ be as in the decomposition \eqref{e:thmFBMdecompostioninproof} where $X_1$ corresponds to the Gaussian component and $X_2$ to the pure L\'evy component. Note that by convexity of $\tau_{X^*}$ we have $\tau_{X^*}(1) \leq \frac{1}{2} \left( \tau_{X^*}(0)+\tau_{X^*}(2) \right) = 1-\frac{\alpha}{2}$. On the other hand, since $\E X^*_2(t)=0$, for $x\in \R$ we have
\begin{equation*}
|x| = \left| x +  \E X^*_2(t) \right| \leq \E \left| x +  X^*_2(t) \right|.
\end{equation*}
By integrating with respect to distribution function of $X^*_1(t)$ one gets
\begin{equation*}
\E \left|X^*_1(t)\right| \leq \E \left|X^*_1(t)+X^*_2(t) \right|
\end{equation*}
and from here it follows that
\begin{equation*}
\tau_{X^*}(1) \geq \tau_{X_1^*}(1).
\end{equation*}
From \eqref{gaussianmarginalsfidis}, we have that $X_1^*(t)$ is Gaussian with zero mean and variance 
\begin{align*}
\E X_1^*(t)^2 &= b \int_0^\infty \int_0^{t}\left( 1 -e^{-\xi (t - s)} \right) ds  \xi^{-1} \pi(d\xi)\\
&=b \int_0^{\infty} \left( 1 -e^{-w} \right) \int_{w/(t)}^{\infty} \xi^{-2} \pi(d\xi) dw\\
&\sim b \frac{\Gamma(1+\alpha)}{(2-\alpha)(1-\alpha)} \ell(t)  t^{2-\alpha} = \widetilde{\sigma}^2 \ell(t)  t^{2-\alpha},
\end{align*}
as $t\to \infty$ by \eqref{slowlyvarying under int}. Since $\E \left|X^*_1(t)\right| = \left( \pi \E X_1^*(t)^2 /2 \right)^{1/2}$, we obtain that $\tau_{X_1^*}(1)=1-\frac{\alpha}{2}$. This proves that $\tau_{X^*}(1) \geq 1-\frac{\alpha}{2}$ and finally $\tau_{X^*}(1) = 1-\frac{\alpha}{2}$.

If $X^*$ is Gaussian, then by using the expression for absolute moments of Gaussian distribution we have for any even integer $q$ as $T\to\infty$
\begin{equation}\label{e:momFBM:part1}
\begin{aligned}
\E \left(A_T^{-1} X^*(Tt)\right)^q &= \frac{q!}{2^{q/2} (q/2)!} \left(A_T^{-2} \E X^*(Tt)^2\right)^{q/2} \to  \frac{q!}{2^{q/2} (q/2)!} \left( \widetilde{\sigma}^2 t^{2-\alpha} \right)^{q/2}\\
&= \E \left(\widetilde{\sigma} B_{1-\alpha/2}(t) \right)^q.
\end{aligned}
\end{equation}
Since we can take $q$ arbitrary large, by \cite[Theorem 1]{GLST2017Arxiv} it follows that $\tau_{X^*}(q)=\left( 1 - \frac{\alpha}{2}\right) q$ for every $q>0$ and hence \eqref{e:FBMtau0} holds.
\end{proof}

\bigskip

\begin{proof}[Proof of Theorem \ref{thm:momSLPcase}]
Suppose first that $q<1+\alpha$ and let $A_T=T^{1/(1+\alpha)} \ell^{\#}\left(T \right)^{1/(1+\alpha)}$. We will show that $\{|A_T^{-1} X^*(Tt)|^q\}$ is uniformly integrable so that 
\begin{equation*}
\E |A_T^{-1} X^*(Tt)|^q \to \E |L_{1+\alpha}(t)|^q \ \text{  as } T\to\infty. 
\end{equation*}
We may assume that $q>1$. We first bound the cumulant function and then use \eqref{momtocumfun-inequality}. Using the notation from the proof of Theorem \ref{thm:SLPcase}, we have from \eqref{e:thmSLP:decomposition}
\begin{equation*}
| \kappa_{X^*} (A_T^{-1} \zeta, Tt) | \leq | \kappa_{\Delta X_{(1)}^*} (A_T^{-1} \zeta, Tt) | + |\kappa_{\Delta X_{(2)}^*} (A_T^{-1} \zeta, Tt)|.
\end{equation*}
Given $q>1$, we may take $\varepsilon$ small enough so that $q<1+\alpha-\varepsilon=:\gamma$ and $\max \{ \beta_{BG}, 1\}<\gamma<1+\alpha$. From \eqref{e:thmSLPproofX1zero} we have that
\begin{equation*}
| \kappa_{\Delta X_{(1)}^*} (A_T^{-1} \zeta, Tt) | \leq  C_1 |\zeta|^{\gamma},
\end{equation*}
and from \eqref{e:splitI1} and \eqref{e:splitI2}
\begin{equation*}
|\kappa_{\Delta X_{(2)}^*} (A_T^{-1} \zeta, Tt)| \leq |I_{T,1}| + |I_{T,2}|.
\end{equation*}
Suppose that $\zeta>0$, the argument is analogous in the other case. Note that $\max \left\{ x^{-\varepsilon} \zeta^\varepsilon , x^{\varepsilon} \zeta^{-\varepsilon} \right\}\leq \max \left\{ \zeta^\varepsilon , \zeta^{-\varepsilon} \right\} \max \left\{ x^{-\varepsilon}, x^{\varepsilon} \right\}$. Using \eqref{e:thmSPLPotterbounds} with $\delta=\varepsilon$, we get the bound
\begin{align*}
\left| I_{T,1} \right| &\leq C_2 \zeta^{1+\alpha} \max \left\{ \zeta^\varepsilon , \zeta^{-\varepsilon} \right\} \int_0^\infty \int_{0}^{t} \left( \overline{\kappa}_{L,1}(x) + \kappa_{L,2}(x) \right) \alpha x^{-\alpha-2} ds dx\\
&\leq C_3 \zeta^{1+\alpha} \max \left\{ \zeta^\varepsilon , \zeta^{-\varepsilon} \right\},
\end{align*}
where $\overline{\kappa}_{L,1}$ is defined in \eqref{kappaL1bar} and the integral is finite by \eqref{kappaL2integrable} and \eqref{kappaL1bar}. For $|I_{T,2}|$ we arrive at the following bound by modifying \eqref{e:thmSLPproofIT2zero}
\begin{align*}
\left| I_{T,2} \right| &\leq C_4 \zeta^{1+\alpha} \int_{0}^{t} \int_0^\infty x^{\gamma-\alpha-2}  \max \left\{ x^{-\delta} \zeta^\delta , x^{\delta} \zeta^{-\delta} \right\} \1_{\left[ \frac{\zeta u T}{ A_T \log 2}, \infty \right)}(x) dx du\\
&= C_4 \zeta^{1+\alpha+\delta} \int_{0}^{t} \int_0^\zeta x^{\gamma-\alpha-2-\delta} \1_{\left[ \frac{\zeta u T}{ A_T \log 2}, \infty \right)}(x) dx du\\
&\qquad \qquad + C_2 \zeta^{1+\alpha-\delta} \int_{0}^{t} \int_\zeta^\infty x^{\gamma-\alpha-2+\delta} \1_{\left[ \frac{\zeta u T}{A_T \log 2}, \infty \right)}(x) dx du\\
&= C_4 \zeta^{\gamma} \int_{0}^{t} \int_0^1 y^{\gamma-\alpha-2-\delta} \1_{\left[ \frac{u T}{A_T \log 2}, \infty \right)}(y) dy du\\
&\qquad \qquad + C_4 \zeta^{\gamma} \int_{0}^{t} \int_1^\infty y^{\gamma-\alpha-2+\delta} \1_{\left[ \frac{u T}{A_T \log 2}, \infty \right)}(y) dy du \leq C_5 \zeta^\gamma,
\end{align*}
with the last inequality coming from the fact that both integrals converge to zero by \eqref{e:thmSLPproofIT2zero}. By combining these bounds, we conclude that
\begin{equation}\label{e:thmmomSLPcasebound}
\begin{aligned}
| \kappa_{X^*} (A_T^{-1} \zeta, Tt) | &\leq C_1 |\zeta|^{\gamma} + C_3 |\zeta|^{1+\alpha} \max \left\{ |\zeta|^\varepsilon , |\zeta|^{-\varepsilon} \right\} + C_5 |\zeta|^\gamma\\
&\leq \begin{cases}
C_6 |\zeta|^\gamma, & \ |\zeta|\leq 1,\\
C_7 |\zeta|^{1+\alpha+\delta}, & \ |\zeta|> 1.\\
\end{cases}
\end{aligned}
\end{equation}
From \eqref{momtocumfun-inequality} we now have
\begin{align*}
\E \left| A_T^{-1} X^*(Tt)\right|^q &\leq k_q \int_{-\infty}^{\infty} \left( 1- \exp  \{ - 2 |\kappa_{X^*} (A_T^{-1} \zeta, Tt) | \} \right) |\zeta|^{-q-1} d \zeta\\
&\leq k_q \int_{|\zeta|\leq 1} \left( 1- \exp  \{ - 2 C_6 \left|\zeta \right|^{1+\alpha - \varepsilon} \} \right) |\zeta|^{-q-1} d \zeta\\
&\qquad \qquad + k_q \int_{|\zeta|>1} \left( 1- \exp  \{ - 2 C_7 \left|\zeta \right|^{1+\alpha + \varepsilon} \} \right) |\zeta|^{-q-1} d \zeta\\
&\leq k_q \int_{-\infty}^{\infty} \left( 1- \exp  \{ - 2 C_6 \left|\zeta \right|^{1+\alpha - \varepsilon} \} \right) |\zeta|^{-q-1} d \zeta\\
&\qquad \qquad + k_q \int_{-\infty}^{\infty} \left( 1- \exp  \{ - 2 C_7 \left|\zeta \right|^{1+\alpha + \varepsilon} \} \right) |\zeta|^{-q-1} d \zeta.
\end{align*}
By \eqref{e:b2}, the terms on the right-hand side are $q$-th absolute moments of $(1+\alpha-\varepsilon)$-stable and $(1+\alpha+\varepsilon)$-stable random variables  with characteristic functions $\exp  \{ - 2 C_6 \left|\zeta \right|^{1+\alpha - \varepsilon} \}$ and $\exp  \{ - 2 C_7 \left|\zeta \right|^{1+\alpha + \varepsilon} \}$, respectively. Since $q<1+\alpha-\varepsilon$, both integrals are finite. This proves uniform integrability, hence the convergence of moments.

We now want to prove \eqref{e:SLPtau} holds. Since the limit process $L_{1+\alpha}(t)$ is self-similar with $H=1/(1+\alpha)$, from \cite[Theorem 1]{GLST2017Arxiv} we conclude that $\tau_{X^*}(q)=q/(1+\alpha)$ for $q<1+\alpha$. By \cite[Proposition 2.1]{GLST2016JSP}, the scaling function is always convex, hence continuous, so that $\tau_{X^*}(1+\alpha)=1$. On the other hand, from \cite[Theorem 7]{GLST2017Arxiv} we have that $\tau_{X^*}(q)=q-\alpha$ for $q\geq 2$. By taking $1+\alpha$ and $q_1, q_2 \geq 2$ with $q_1<q_2$, we find that $\tau_{X^*}(q)=q-\alpha$ for $q\in \{1+\alpha, q_1, q_2\}$. Hence, these three points lie on a straight line and $\tau_{X^*}$ must be linear on $[1+\alpha,q_2]$ by \cite[Lemma 3]{GLST2017Arxiv}. On the other hand, $\tau_{X^*}(q)=q-\alpha$ for any $q\geq1+\alpha$, which completes the proof of \eqref{e:SLPtau}. 
\end{proof}

\bigskip

\begin{proof}[Proof of Theorem \ref{thm:momZcase}]
The proof is completely analogous to the proof of Theorem \ref{thm:momSLPcase}. In \eqref{e:Zcasemomentbound} and \eqref{e:6.8b} we have already derived the following bound for the $q$-th absolute moment 
\begin{equation*}
\E \left| A_T^{-1} X^*(Tt)\right|^q \leq k_q \int_{-\infty}^{\infty} \left( 1- \exp  \{ - 2 C_2 t^{\beta-\alpha-\delta} \left|\zeta \right|^{\beta} \} \right) |\zeta|^{-q-1} d \zeta,
\end{equation*}
with the integral on the right finite. By \eqref{e:b2}, it is $q$-th absolute moment of a symmetric $\beta$-stable random variable. To prove that \eqref{e:Ztau} holds, one proceeds as in the end of the proof of Theorem \ref{thm:momSLPcase}.
\end{proof}

\bigskip

\begin{proof}[Proof of Theorem \ref{thm:momBMcase}]
If $\mu_L \not\equiv 0$, then $X$ is non-Gaussian and by \cite[Theorem 7]{GLST2017Arxiv} we have that $\tau_{X^*}(q)=q-\alpha$ for $q\geq q^*$, where $q^*$ is the smallest even integer greater than $2\alpha$. 

We next establish the asymptotic behavior of even moments of order less than $2\alpha$. Note that $2 \alpha$ is an even integer and let $\kappa_Y^{(m)}$ denote the $m$-th order cumulant of random variable $Y$:
\begin{equation*}
\kappa_Y^{(m)} = (-i)^m \frac{d^m}{d\zeta^m} \kappa_Y(\zeta) \big|_{\zeta=0}.
\end{equation*}
For a stochastic process $Y=\{Y(t)\}$ we write $\kappa_Y^{(m)}(t)=\kappa_{Y(t)}^{(m)}$, and by suppressing $t$ we mean $\kappa_Y^{(m)}=\kappa_Y^{(m)}(1)$. From the assumption of analyticity of $\kappa_{X}$ around origin, we have by \cite[Theorem 4.2]{bn2001} for $m\in \N$
\begin{equation*}
\kappa_{X^*}^{(m)} (Tt) = m \kappa_{X}^{(m)} I_{m-1}(Tt),
\end{equation*}
where 
\begin{equation}\label{e:Im-1}
I_{m-1}(T) = \int_0^\infty \left( T t \xi + \sum_{k=1}^{m-1} (-1)^{k-1} {m-1 \choose k} \frac{1}{k} \left( e^{-ktT \xi} - 1 \right) \right) \xi^{-m} \pi(d\xi).
\end{equation}
Suppose that $m<\alpha+1$. The function under the integral in \eqref{e:Im-1} is bounded by $C \xi^{-m+1}$ and $\int_0^\infty \xi^{-m+1} \pi(d\xi)<\infty$, hence we can apply the dominated convergence theorem to conclude that 
\begin{equation}\label{e:proofofthmBMmomJi-1}
\frac{I_{m-1}(T)}{tT} \to J_{m-1}:= \int_0^\infty \xi^{-m+1} \pi(d\xi)<\infty,
\end{equation}
and so 
\begin{equation}\label{e:kmT}
\kappa_{X^*}^{(m)}(Tt) \sim m \kappa_{X}^{(m)} J_{m-1} Tt.
\end{equation}
On the other hand, for $m>\alpha+1$ such that $\kappa_{X}^{(m)} \neq 0$ we have by \cite[Lemma 2]{GLST2017Arxiv} that 

\begin{equation}\label{e:m-a}
\kappa_{X^*}^{(m)} (Tt) \sim \ell_m(Tt) (Tt)^{m-\alpha}
\end{equation}
for some slowly varying function at infinity $\ell_m$. For $m=\alpha+1$, if $\int_0^\infty \xi^{-m+1} \pi(d\xi)<\infty$, then \eqref{e:proofofthmBMmomJi-1} still holds. If on the other hand $\int_0^\infty \xi^{-m+1} \pi(d\xi)=\infty$, we can, as in the proof of \cite[Lemma 2]{GLST2017Arxiv}, show that for any $\varepsilon>0$, $T$ can be taken large enough so that $|\kappa_{X^*}^{(\alpha+1)}(T)| \leq C T^{1+\varepsilon}$.

We now have to go from cumulants to moments. Let $m$ be an even integer $m \in \{2,\dots, 2\alpha-2\}$. Since $\mu_L \not\equiv 0$, by \cite[Remark 3.4.]{gupta2009cumulants} we have that $\kappa_X^{(m)} \neq 0$ for every even $m$. Using the expression for moment in terms of cumulants (see e.g.~\cite[Proposition 3.3.1]{peccati2011wiener}), for an even integer $m$ we have
\begin{equation}\label{proof:cor:intermittency-integrated}
\E|X^*(Tt)|^m = \E(X^*(Tt))^m = \sum_{k=1}^m B_{m,k} \left( \kappa_{X^*}^{(1)}(Tt), \dots, \kappa_{X^*}^{(m-k+1)}(Tt) \right),
\end{equation}
where $B_{m,k}$ is the partial Bell polynomial given by (see \cite[Definition 2.4.1]{peccati2011wiener})
\begin{equation*}
B_{m,k}(x_1,\dots,x_{m-k+1}) = \sum_{r_1,\dots,r_{m-k+1}} \frac{m!}{r_1! \cdots r_{m-k+1}!} \left(\frac{x_1}{1!}\right)^{r_1} \cdots \left(\frac{x_{m-k+1}}{(m-k+1)!}\right)^{r_{m-k+1}}
\end{equation*}
and the sum is over all nonnegative integers $r_1,\dots,r_{m-k+1}$ satisfying $r_1+\cdots +r_{m-k+1}=k$ and
\begin{equation}\label{proof:cor:intermittency-integrated3}
1 r_1 + 2 r_2 + \cdots + (m-k+1) r_{m-k+1}=m.
\end{equation}
Since $\kappa_{X^*}^{(1)}(Tt)=0$, the nonzero terms of the sum in the expression for $B_{m,k} \big(  \kappa_{X^*}^{(1)}(Tt), \allowbreak \dots,\allowbreak \kappa_{X^*}^{(m-k+1)}(Tt)\allowbreak \big)$ are obtained when $r_1=0$. 

Suppose first that $m<1+\alpha$ so that $\kappa_{X^*}^{(l)}(Tt) \sim \kappa_{X}^{(l)} l J_{l-1} Tt$ for every $l\in \{2,\dots,m\}$ by \eqref{e:kmT}. It is easy to see that the highest power of $T$ in \eqref{proof:cor:intermittency-integrated} will then be obtained by taking one of the $r_i$'s as large as possible. This is obviously achieved by taking $r_2=m/2$ which implies $r_i=0$ for $i\neq2$ by \eqref{proof:cor:intermittency-integrated3}. We conclude that 
\begin{equation}\label{e:thmBMmomasymptot}
\E|X^*(T)|^m \sim \frac{m!}{(m/2)!} \left( \frac{2 \kappa_{X}^{(2)}  J_{1}}{2} \right)^{m/2}  (Tt)^{m/2} = \frac{m!}{2^{m/2} (m/2)!} \widetilde{\sigma}^m  (Tt)^{m/2},
\end{equation}
with $\widetilde{\sigma}$ defined in \eqref{sigmatilBM}. 

Now if $m>\alpha+1$, we may have additional terms of the form $\ell_l(Tt) (Tt)^{l-\alpha}$, $l\in \{ \lceil \alpha+1 \rceil,\dots, m\}$ or the one coming from $\kappa_{X^*}^{(\alpha+1)}(Tt)$. Since $|\kappa_{X^*}^{(\alpha+1)}(Tt)| \leq C T^{1+\varepsilon}$ for $\varepsilon$ arbitrary small, the highest power of $T$ coming from these terms would correspond to the term $\ell_m(Tt) (Tt)^{m-\alpha}$ by \eqref{e:m-a}. However, since $m<2\alpha \Leftrightarrow m-\alpha<m/2$, this would not dominate the term with $T^{m/2}$ that can be obtained as in the previous case and hence \eqref{e:thmBMmomasymptot} still holds. To summarize, we have proved that \eqref{e:thmBMmomasymptot} holds for every even integer $m \in \{2,\dots, 2\alpha-2\}$, hence the convergence of moments $\E|T^{-1/2} X^*(Tt)|^q \to \E|B(t)|^q$ holds for all moments of order $q\leq 2\alpha-2$ and every $t>0$. By \cite[Theorem 1]{GLST2017Arxiv}, we have then $\tau_{X^*}(q)=q/2$ for $q\leq 2\alpha-2$. 

It remains to extend the argument to $q\leq 2\alpha$, that is to show that $\tau_{X^*}(q)=q/2$ for $q\leq 2\alpha$. For $m=2\alpha$ we would have in \eqref{proof:cor:intermittency-integrated} the term $\ell_m(Tt) (Tt)^{m-\alpha}=\ell_m(Tt) (Tt)^{\alpha}$ coming from $\kappa_{X^*}^{(m)} (Tt)$ as in \eqref{e:m-a}, and the term of the order $T^{m/2}=T^{\alpha}$ coming from $(\kappa_{X^*}^{(2)} (Tt))^{m/2} = (\kappa_{X^*}^{(2)} (Tt))^{\alpha}$ as in \eqref{e:kmT}. The exact asymptotics as in \eqref{e:thmBMmomasymptot} would depend on the form of the slowly varying function $\ell_m$, but nevertheless it can be represented as $\E|X^*(T)|^m \sim \widetilde{\ell}_m(T) T^\alpha$ for some slowly varying function $\widetilde{\ell}_m$. From \eqref{deftau} we conclude that $\tau_{X^*}(2\alpha)=\alpha$.

Consider now three points $q_1=0$, $0<q_2\leq 2\alpha-2$ and $q_3=2\alpha$. We have proved that $\tau_{X^*}(q)=q/2$ for $q \in \{q_1, q_2, q_3\}$. Since the scaling function is always convex (\cite[Proposition 2.1]{GLST2016JSP}) and the convex function passing through three collinear points is linear (\cite[Lemma 3]{GLST2017Arxiv}), we conclude that $\tau_{X^*}(q)=q/2$ for $q\leq 2\alpha$. On the other hand, if we take $q_1=2\alpha$ and $q^*\leq q_2<q_3$, then $\tau_{X^*}(q)=q-\alpha$ for $q \in \{q_1, q_2, q_3\}$. Since $q_3$ can be taken arbitrarily large, by using convexity again we conclude $\tau_{X^*}(q)=q-\alpha$ for $q\geq 2\alpha$.

If $X^*$ is Gaussian, then all the cumulants of order greater than $2$ are zero and hence \eqref{e:thmBMmomasymptot} would hold for any even integer $m$. This implies then that $\tau_{X^*}(q)=q/2$ for every $q>0$.
\end{proof}

\bigskip


\section*{Acknowledgements}
Nikolai N.~Leonenko was supported in particular by Cardiff Incoming Visiting Fellowship Scheme, International Collaboration Seedcorn Fund, Australian Research Council's Discovery Projects funding scheme (project DP160101366)and the project MTM2015-71839-P of MINECO, Spain (co-funded with FEDER funds). Murad S.~Taqqu was supported in part by the Simons foundation grant 569118 at Boston University.

\bigskip
\bigskip

\bibliographystyle{agsm}
\bibliography{References}

\end{document}